\documentclass[onefignum,onetabnum]{siamart171218}

\usepackage{graphicx,color}
\usepackage{hyperref}
\usepackage{dsfont}
\usepackage{float}
\usepackage{capt-of}
\usepackage{mathtools}
\usepackage{amssymb}
\usepackage{stmaryrd}
\usepackage{lipsum}
\usepackage{amsfonts}
\usepackage{amsmath}
\usepackage{graphicx}
\usepackage{epstopdf}
\usepackage{algorithmic}
\ifpdf
  \DeclareGraphicsExtensions{.eps,.pdf,.png,.jpg}
\else
  \DeclareGraphicsExtensions{.eps}
\fi

\newcommand{\R}{\mathbb{R}}
\newcommand{\N}{\mathbb{N}}

\newsiamremark{remark}{Remark}
\newsiamremark{hypothesis}{Hypothesis}
\crefname{hypothesis}{Hypothesis}{Hypotheses}
\newsiamthm{claim}{Claim}

\headers{Global uniqueness and Lipschitz-stability}{B.  Harrach and  H.  Meftahi}

\title{Global uniqueness and Lipschitz-stability for the inverse  Robin transmission  problem%
\thanks{{\bf Funding:} For the second author, this work was funded by the German Academic Exchange Service (DAAD)}}

\author{B.  Harrach\thanks{Department of Mathematics, Goethe University Frankfurt, Germany 
  \email{harrach@math.uni-frankfurt.de}.}
\and H.  Meftahi\thanks{Department of Mathematics, ENIT of Tunisia
  \email{houcine.meftahi@enit.utm.tn}.}}

\usepackage{amsopn}

\def\norm#1{\hspace{0.2ex} \|#1\| \hspace{0.2ex}} 

\begin{document}
\maketitle

\let\thefootnote\relax\footnotetext{\hrule \vspace{1ex} \centering This is a preprint version of a journal article published in\\
 \emph{SIAM J.\ Appl.\ Math.} \textbf{79}(2), 525--550, 2019
(\url{https://doi.org/10.1137/18M1205388}).
}

\begin{abstract}
In this paper, we consider the inverse problem of detecting a corrosion coefficient between two layers of a conducting medium from the Neumann-to-Dirichlet map. 
This inverse problem is motivated by the description of the index of corrosion in non-destructive testing. We show a monotonicity estimates between the Robin coefficient and the Neumann-to-Dirichlet operator. We prove a global uniqueness result and Lipschitz stability estimate, and show how to quantify the Lipschitz stability constant for a given setting. 

Our quantification of the Lipschitz constant does not rely on quantitative unique continuation or analytic estimates of special functions.
Instead of deriving an analytic estimate, we show that the Lipschitz constant for a given setting can be explicitly calculated from the a priori data by solving finitely many well-posed PDEs. Our arguments rely on standard (non-quantitative) unique continuation, a Runge approximation property, the monotonicity result and the method of localized  potentials. 

To solve the problem numerically, we reformulate the inverse problem into a minimization problem using a least square functional.
The reformulation of the minimization problem as a suitable saddle point problem
allows us to obtain the optimality conditions by using differentiability properties of the min-sup formulation. The reconstruction is then performed by means of the BFGS algorithm. Finally, numerical results are presented  to illustrate the efficiency of the proposed 
 alogorithm.
 \end{abstract}

\section{Introduction}
This paper is concerned with the inverse problem of detecting a corrosion contamination between two layers of a non-homogenous electric conductor. 
This problem can be encountered  in several area of engineering such as  diffusion of chemical substances in a given medium, delamination in certain elastic materials \cite{amaya2005mathematical,Hu20021479}.
 
The corrosion may occur in  many different  forms and several models  are considered in the literature \cite{MR3298683,MR3636710,MR3574909,choulli1,choulli2,MR2079585,MR2507508}.  Identifying the Robin parameter from boundary measurements turns out to be 
a way  to locate  the corroded part  in a given medium and possibly evaluate the damage level by electrical impedance tomography process.   

In this work, we consider the mathematical model problem  where the corrosion takes place between two layers of a non-homogenous medium \cite{MR3298683,MR3636710,MR3574909}. The geometry of the corrosion boundary is assumed to be known  in advance, but  the coefficient of corrosion is unknown and is subject 
to be reconstructed from the Neumann-to-Dirichlet  map. This problem shares similarities with the inverse problem with Robin condition on the external boundary \cite{choulli1,choulli2,MR2079585,MR2507508,choulli2018inverse}.

The two major questions for the inverse problem are the uniqueness and stability of a solution. For the classical mathematical model problem,  Chaabane and Jaoua \cite{chaabane1999identification} proved a 
uniqueness result  and    local and monotone Lipschitz stability  estimate from  boundary measurements, provided that the Robin coefficient is a continuous function with some negative lower bound.  The proof rely on the study of the behavior  of the solution of the forward problem with respect to the Robin coefficient.
In \cite{sincich2007lipschitz} Sincich established a Lipschitz stability estimate  from electrostatic  boundary  measurements  under a further prior assumptions  of a piecewise constant Robin coefficient.  The Lipschitz  constant behaves exponentially with respect to the portions considered.

For the model problem considered here, a particular  uniqueness result for  simultaneous reconstruction of the conductivity  and the Robin coefficient  is proven in \cite{MR3298683}. The proof is based on integrals representation of the solution of the forward problem  and decomposition on the basis of spherical harmonics.  

Several numerical methods have been proposed for the inverse Robin problem in the context of  external boundary corrosion detection \cite{MR1463588,MR1675331,MR2087981,MR2079585,MR2929616,MR3311476}. Some methods are based on the variational approach,  such as the Kohn-Vogelius  functional \cite{MR2079585,MR2929616},  the least square functional
\cite{MR3311476}. $L^1$-tracking functional  is considered  in \cite{MR2087981}, where the authors  prove differentiability using complex analysis techniques.  The proof is strongly related to the positivity and monotonicity of  of the derivative of the state.  This kind of functional is robust with respect to data outliers  but it has not been tested numerically.  Some numerical results restricted  to the case of thin plate  are  based on the asymptotic expansion of the solution \cite{MR1463588,MR1675331}.

In this paper we prove a global uniqueness  and Lipschitz stability estimate for the Robin transmission inverse problem, and show how to explicitly calculate the Lipschitz stability constant for a given setting by solving finitely many well-posed PDEs. The proof is based on a monotonicity estimates combined with the method of localized potentials \cite{gebauer2008localized} that we derive from a Runge approximation result. For the numerical solution of the Robin transmission inverse problem, we reformulate the inverse problem into a minimization problem using a least square functional, and use a quasi-Newton method which employs the analytic gradient of the cost function and the approximation of the inverse Hessian is updated by the Broyden, Fletcher, Goldfarb, Shanno (BFGS) scheme \cite{murea2005bfgs}. 

Let us give some more remarks on the relation of this work to previous results. Monotonicity estimates and localized potentials techniques  have been used in different ways for the study of inverse problems \cite{harrach2009uniqueness,harrach2010exact,harrach2012simultaneous,arnold2013unique,harrach2013monotonicity,barth2017detecting,harrach2017local,brander2018monotonicity,griesmaier2018monotonicity,harrach2018helmholtz,harrach2018fractional,harrach2018localizing,seo2018learning} and several recent works build practical reconstruction methods on monotonicity properties \cite{tamburrino2002new,harrach2015combining,harrach2015resolution,harrach2016enhancing,maffucci2016novel,tamburrino2016monotonicity,garde2017comparison,garde2017convergence,garde2017regularized,su2017monotonicity,ventre2017design,harrach2018monotonicity,zhou2018monotonicity}. But, together with \cite{harrach2019uniqueness,seo2018learning}, this is the first work proving a Lipschitz stability result with the relatively simple technique of monotonicity and localized potentials. 

Lipschitz stability for inverse coefficient problems has been studied intensively in the literature, cf.\ \cite{kazemi1993stability,alessandrini1996determining,imanuvilov1998lipschitz,imanuvilov2001global,cheng2003lipschitz,alessandrini2005lipschitz,bacchelli2006lipschitz,bellassoued2006lipschitz,klibanov2006lipschitz,klibanov2006lipschitz_nonstandard,bellassoued2007lipschitz,sincich2007lipschitz,yuan2007lipschitz,yuan2009lipschitz,beretta2011lipschitz,beretta2013lipschitz,melendez2013lipschitz,alessandrini2017lipschitz,beretta2017uniqueness,beilina2017lipschitz,alessandrini2018lipschitz,ruland2018lipschitz}. 
Lipschitz stability results are usually based on technically challenging but constructive approaches involving Carleman estimates or quantitative unique continuation. For some applications these constructive approaches also allowed to quantify the asymptotic behavior of the stability constants (cf., e.g., \cite[Corollary 2.5]{sincich2007lipschitz}). But, to the knowledge of the authors, no previous work has derived a method
to explicitly determine a Lipschitz stability constant for a given setting. 

Our approach on proving Lipschitz stability result differs from these previous works. We first prove an abstract Lipschitz stability result (in Theorem 2.1) by relatively simple, but non-constructive arguments that are based on standard (non-quantitative) unique continuation, a Runge approximation property, the monotonicity result and the method of localized potentials. 
Our new approach can easily be extended to other inverse coefficient problems, and it has 
already been used to prove uniqueness and Lipschitz stability 
in Electrical Impedance Tomography with finitely many electrodes \cite{harrach2019uniqueness}, and to study stability
in machine learning reconstruction algorithms \cite{seo2018learning}.

To quantify the Lipschitz stability constant, we then develop a new method (in Theorem 5.2) 
that allows to explicitly calculate the Lipschitz constant for a given setting by solving a finite number of well-posed PDEs. 
Again, we do not use quantitative unique continuation arguments and do not require analytic estimates of special functions. We still work with special solutions (the localized potentials), but instead of deriving analytic expressions and estimates, we show how to calculate certain localized potentials by solving a finite number of PDEs, and then show how to calculate the Lipschitz constant from these solutions.
This might be considered a new paradigm for deriving Lipschitz stability constants that is less elegant than previous results since we do not obtain any analytic estimates. However, our new approach allows to explicitly calculate Lipschitz stability constants for a given setting which may be important to quantify the achievable resolution and noise robustness in practical applications.
  
The paper is organized as follows. In section 2, we introduce the forward and inverse problem,  the Neumann-to-Dirichlet operator and formulate our main theoretical results: a global   uniqueness result and a Lipschitz stability estimate.  Section 3 and 4 contain the main theoretical tools for this work. In  section 3, we  prove a Runge approximation result and based on this result,  we prove the global uniqueness theorem.  In section 4, we show a monotonicity result between the Robin coefficient and the Neumann-to-Dirichlet operator and deduce the existence of localized potential from the Runge approximation result. Then, we prove the Lipschitz stability estimate. 
Section 5 shows how to calculate the Lipschitz stability constant for a given setting.   
 In section 6, we introduce the minimization problem,  we prove the existence of optimal solution and we compute the first order optimality  condition using the framework of the min-sup differentiability. In the last section, satisfactory numerical  results for two-dimensional  problem are presented to illustrate the efficiency of the method.

\section{Problem statement}
Let $\Omega \subset \R^d$ ($d\geq 2$),  be a bounded  domain with  Lipschitz boundary $\partial \Omega$  and let  $\Omega_1$  be an open subset of $\Omega$ such that $\Omega_1\Subset \Omega$, $\Gamma:=\partial\Omega_1$ is Lipschitz, and $\Omega_2:=\Omega\setminus\overline{\Omega_1}$ is connected. Thus, we have $\partial\Omega_2=\Gamma \cup\partial\Omega$; see Figure \ref{geom} for a description of the geometry.
\begin{figure}[ht]
 \centering
 \includegraphics[width=0.4\textwidth]{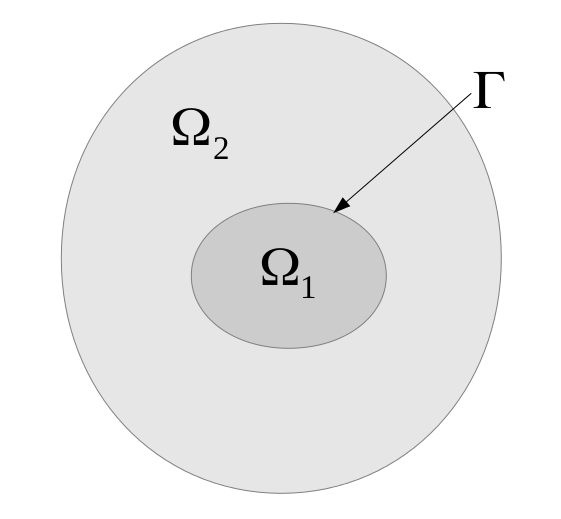}
 \caption{The domain $\Omega=\Omega_{1}\cup\Gamma\cup\Omega_{2}$.}
  \label{geom}
 \end{figure}
For a piecewise-constant conductivity  $\sigma=\sigma_1\chi_{\Omega_1}+\sigma_2\chi_{\Omega_2}$, with given $\sigma_1,\sigma_2>0$, and a Robin parameter 
$\gamma\in L^\infty_+(\Gamma)$, where $L^\infty_+$ denotes the subset of $L^\infty$-functions with positive essential infima, 
 we consider the following  problem  with Neumann boundary data  $g\in L^{2}(\partial\Omega)$:
\begin{equation}
\label{transm}
\left\{
\begin{aligned}
-\text{div}(\sigma \nabla u) & = 0\quad \text{ in }\Omega_1\cup \Omega_2,
\\
 \sigma_2\partial_{\nu} u&= g\quad \text{ on  }\partial \Omega,\\
 \llbracket u\rrbracket&=  0 \quad\text{ on }  \Gamma,\\
 \llbracket \sigma\partial_{\nu} u\rrbracket &= \gamma u \quad\text{ on }  \Gamma,
\end{aligned}
\right.
\end{equation}
where  $\nu$ is the  unit normal vector to the  interface  $\Gamma$  or $\partial\Omega$ pointing  outward of $\Omega_{1}$ or $\Omega$, respectively,  and
\[
\llbracket u \rrbracket:= u^+-u^-, \text{ with } u^+ = \beta_{0}\big|_{\Omega_2}(u),\quad  u^- = \beta_{0}\big|_{\Omega_1}(u),
 \]
 \[
\llbracket \sigma\partial_{\nu}u \rrbracket:=\sigma_2\partial_{\nu}u^+-\sigma_1\partial_{\nu}u^-, \text{ with } \partial_\nu u^+ = \beta_{1}\big|_{\Omega_2}(u),\quad  \partial_\nu u^- = \beta_{1}\big|_{\Omega_1}(u).
 \]
Here $\beta_{0}\big|_{\Omega_i} : H^1(\Omega_{i})\rightarrow L^{2}(\Gamma)$,   $\beta_{1}\big|_{\Omega_i} : H(\Omega_{i})\rightarrow  L^{2}(\Gamma)$ are the Dirichlet and Neumann trace operators and $H(\Omega_{i})$ is a subspace  of $H^1(\Omega_{i})$ defined by 
 \[
 H(\Omega_{i}):=\left\{ u\in H^1(\Omega_{i}):  \Delta u\in L^2(\Omega_{i}) \right\}.
 \]
 In  \eqref{transm},  $u$ represents the electric potential, $g$ the prescribed current density on the  external boundary $\partial\Omega$,  while $\gamma$ is called the corrosion 
 coefficient  and  describes the presence of corrosion damage on $\Gamma$. 
 
The Neumann problem \eqref{transm} is equivalent to the variational formulation of finding 
$u\in H^1(\Omega)$ sucht that
\begin{equation}
\label{eqv}
\int_\Omega\sigma\nabla u\cdot\nabla w\,dx+\int_\Gamma \gamma uw\,ds=\int_{\partial\Omega}gw\,ds \text{ for all } w\in  H^1(\Omega).
\end{equation}
Using the Riesz representation theorem (or the Lax-Milgram-Theorem), it is easily seen that 
\eqref{eqv} is uniquely solvable and that the solution depends continuously on $g\in L^2(\partial \Omega)$ and $\gamma\in L^\infty_+(\Omega)$.
When dealing with different Robin coefficients or Neumann data, we also denote the solution by $u_\gamma^{(g)}$.

 We  denote by   $\Lambda(\gamma)$  the so-called Neumann-to-Dirichlet  map:
 \[
 \begin{aligned}
 \Lambda(\gamma) : L^{2}(\partial \Omega)& \longrightarrow  L^{2}(\partial \Omega)\\
   g&\longmapsto u_{|\partial \Omega}, 
 \end{aligned}
 \]
  where $u$ is solution to \eqref{transm}. The physical interpretation of the Neumann-to-Dirichlet  (ND) map $\Lambda(\gamma)$ is knowledge of   the resulting voltages distributions  on the boundary   of $\Omega$ corresponding to all possible current distributions on the boundary. Thus the  inverse problem we are concerned with  is the following:
\begin{equation}
\label{invp}
\text{ \it  Find the parameter }  \gamma\quad \text{ \it  from  the knowledge of the  map  } \quad \Lambda(\gamma).  
 \end{equation}
  There are several aspects of this inverse problem which are interesting both for mathematical theory and practical applications.
 \begin{itemize}
 \item {\it Uniqueness.}  If $\Lambda(\gamma_1)=\Lambda(\gamma_2)$ show that $\gamma_1=\gamma_2$.
 \item  {\it Stability.} If  $\Lambda(\gamma_1)$  is close to  $\Lambda(\gamma_2)$,
 show that $\gamma_1$ and $\gamma_2$ are close (in a suitable sense).
 \item  {\it Reconstruction.} Given the boundary measurements $\Lambda_\gamma$, find 
 a procedure to reconstruct the Robin parameter $\gamma$.
 \item {\it  Partial data.} If $S$ is a subset of $\partial\Omega$ and if
  $\Lambda(\gamma_1) g|_S=
 \Lambda(\gamma_2)g|_S $ for all boundary currents $g$, show that $\gamma_1=\gamma_2$.
 \end{itemize}
 
 In this work, we will concentrate on the first three aspects. Note however that our theoretical results on uniqueness and stability also hold with the same proofs when
 currents are applied on an arbitrarily small open subset $S\subseteq \partial \Omega$ (with $\partial \Omega\setminus S$ kept insulated) 
 and voltages are measured on the same subset $S$.

We will show  a monotonicity result: 
\[
\gamma_1\leq \gamma_2 \quad \text{implies }\quad \Lambda(\gamma_1)\geq \Lambda(\gamma_2)\quad \text{in the sense of quadratic forms},
\]
and using Runge approximation and localized potentials, we deduce the following  uniqueness  and stability results  for determining $\gamma$ from $\Lambda(\gamma)$.
\begin{theorem}[Uniqueness]
\label{uniqueness}
For $\gamma_1,\gamma_2\in L^\infty_+(\Gamma)$, 
\[
\Lambda(\gamma_1)=\Lambda(\gamma_2)\quad \text{if and only if}\quad \gamma_1=\gamma_2.
\] 
\end{theorem}
 \begin{theorem}[Lipschitz stability]
\label{stability}
Let $\mathcal{A}$ be a finite dimensional subspace of $L^\infty(\Gamma)$. Given constants $b>a>0$, we define
\[
\mathcal{A}_{[a,b]}:=\left\{\gamma\in \mathcal{A}:\quad   a\leq\gamma(x)\leq b \text{ for } x\in \Gamma\ \text{(a.e.)}\right\}.   
\]
Then there exists a constant $C>0$ such that   for all $\gamma_1,\gamma_2\in \mathcal{A}_{[a,b]}$, we have 
\[
\| \gamma_1-\gamma_2 \|_{L^{\infty}(\Gamma)}\leq   C\| \Lambda(\gamma_1)-\Lambda(\gamma_2) \|_{*},
\]
where $\| . \|_{*}$ is the natural operator norm in $\mathcal L(L^2(\partial \Omega))$.
\end{theorem}
Theorem \ref{uniqueness} and \ref{stability} will be proven in the following two sections. Note that Theorem \ref{stability}
obviously implies Theorem \ref{uniqueness} by setting $\mathcal 
A:=\mathrm{span}\{\gamma_1,\gamma_2\}$ and choosing $a,b>0$ so that 
$\gamma_1,\gamma_2\in \mathcal A_{[a,b]}$. Nevertheless we give an 
independent proof of Theorem \ref{uniqueness} since it is short and simple.

\section{Runge approximation and uniqueness}

We will deduce the uniqueness theorem \ref{uniqueness} from the following Runge approximation result:

\begin{theorem}[Runge approximation]\label{thm:runge}
Let $\gamma\in L^\infty_+(\Gamma)$. For all $f\in L^2(\Gamma)$ there exists a sequence $(g_n)_{n\in\N}\subset L^2(\partial\Omega)$ such that the corresponding solutions $u^{(g_n)}$ of \eqref{transm} with boundary data $g_n$, $n\in \N$, fulfill
\[
u^{(g_n)}|_\Gamma\to f \quad \text{ in } L^2(\Gamma).
\]
\end{theorem}
\begin{proof}
We introduce the operator
\[
A:  L^2(\Gamma)\to  L^2(\partial\Omega), \quad f\mapsto A f:= v_{|\partial\Omega}, 
\]
where $v\in H^1(\Omega)$  solves 
\begin{equation}\label{eq:runge_v}
\int_\Omega\sigma\nabla v\cdot\nabla w\,dx+\int_\Gamma\gamma v w\,ds=\int_{\Gamma}f w\,ds\quad \text{ for all } w\in H^1(\Omega).
\end{equation}

Let  $g\in L^2(\partial \Omega)$ and $u\in H^1(\Omega)$ be the corresponding solution  of problem \eqref{transm}.  Then the adjoint operator of  $A$ is characterized by 
\begin{align*}
\int_{\Gamma} \left( A^* g \right) f \,ds 
&= \int_{\partial \Omega} \left( A f \right) g \,ds 
= \int_{\partial \Omega} v g \,ds 
=\int_\Omega \nabla u\cdot\nabla v\,dx+\int_\Gamma \gamma uv\,ds\\
&=\int_{\Gamma}f u\,ds, \quad \text{ for all } f\in L^2(\Gamma),
\end{align*}
which shows that $A^*:\ L^2(\partial \Omega)\to L^2(\Gamma)$ fulfills 
$A^* g=u|_{\Gamma}$. The assertion follows if we can show that $A^*$ has dense range, which is equivalent
to $A$ being injective.

To prove this, let $v|_{\partial \Omega}=Af=0$ with $v\in H^1(\Omega)$ solving \eqref{eq:runge_v}.
Since \eqref{eq:runge_v} also implies that $\sigma_2\partial_\nu v|_{\partial \Omega}=0$, and $\Omega_2$ is connected, it follows by unique continuation that $v|_{\Omega_2}=0$ and thus $v^+|_{\Gamma}=0$. Since $v\in H^1(\Omega)$ this also implies that
$v^-|_{\Gamma}=0$, and together with \eqref{eq:runge_v} we obtain that $v|_{\Omega_1}\in H^1(\Omega_1) $ solves
\[
\nabla\cdot (\sigma_1 \nabla v)=0 \quad \text{ in } \Omega_1
\]
with homogeneous Dirichlet boundary data $v|_{\partial \Omega_1}=0$. Hence, $v|_{\Omega_1}=0$, so that $v=0$ almost everywhere in $\Omega$.
From \eqref{eq:runge_v} it then follows that $\int_{\Gamma}f w\,ds=0$ for all $w\in H^1(\Omega)$ and thus $f=0$.
\end{proof}

\begin{proof}[Proof of Theorem \ref{uniqueness}]
For Robin parameters $\gamma_1,\gamma_2\in L^\infty_+(\Gamma)$ and Neumann data $g,h\in L^2(\partial \Omega)$ we denote the
corresponding solutions of \eqref{transm} by $u_{1}^g$, $u_1^h$, $u_2^g$, and $u_2^h$ respectively. 

The variational formulation \eqref{eqv} yields the Alessandrini-type equality
\begin{align*}
\lefteqn{\int_{\partial \Omega} h \left( \Lambda(\gamma_2)-\Lambda(\gamma_1)\right) g \, ds}\\
&= \int_{\partial \Omega} h \Lambda(\gamma_2) g \, ds  - \int_{\partial \Omega} g \Lambda(\gamma_1) h \, ds
= \int_{\partial \Omega} h u_2^g \, ds  - \int_{\partial \Omega} g u_1^h \, ds\\
&= \int_\Omega\sigma\nabla u_1^h \cdot\nabla u_2^g \,dx+\int_\Gamma \gamma_1 u_1^h u_2^g\, ds
- \left( \int_\Omega\sigma\nabla u_2^g \cdot\nabla u_1^h \,dx+\int_\Gamma \gamma_2 u_2^g u_1^h\, ds\right)\\
&= \int_\Gamma (\gamma_1-\gamma_2) u_1^h u_2^g\, ds. 
\end{align*}
This shows that $\Lambda(\gamma_1)=\Lambda(\gamma_2)$ implies that
\[
\int_\Gamma (\gamma_1-\gamma_2) u_1^h u_2^g\, ds=0,  \quad \text{ for all } g,h\in L^2(\partial \Omega).
\]
Using the Runge approximation result in theorem~\ref{thm:runge}, this yields that 
$(\gamma_1-\gamma_2) u_1^h=0$ (a.e.) in $\Gamma$ for all $h\in L^2(\partial \Omega)$, and using theorem~\ref{thm:runge} again, this implies
$\gamma_1=\gamma_2$. 
\end{proof}


\section{Monotonicity, localized potentials and Lipschitz stability}

To prove the Lipschitz stability result in Theorem \ref{stability}, we first show a monotonicity estimate between the Robin coefficient and the Neumann-to-Dirichlet operator, and deduce the existence of localized potentials from the Runge approximation result in the last section.
The idea of the following monotonicity estimate stems from the works of Ikehata, Kang, Seo, and Sheen \cite{ikehata1998size,kang1997inverse},
and the proof is analogue to \cite[Lemma 2.1]{harrach2010exact}, cf.\ also the other references to monotonicity-based methods in the introduction.

\begin{lemma}[Monotonicity estimate]
\label{mono}
Let $\gamma_1, \gamma_2\in L^\infty_+(\Gamma)$ be two Robin parameters, let $g\in L^2(\partial\Omega)$ be an applied boundary current, and let $u_2:=u^{g}_{\gamma_2}\in H^1(\Omega)$ solve \eqref{transm} for the boundary current $g$ and the Robin parameter $\gamma_2$. Then
\begin{equation}
\label{eqmono}
\int_\Gamma(\gamma_1-\gamma_2)u_2^2\,ds\geq \int_{\partial \Omega} g \left(\Lambda(\gamma_2)-\Lambda(\gamma_1)\right) g\, ds\geq 
\int_\Gamma\left(\gamma_2-\frac{\gamma^2_2}{\gamma_1}  \right)u^2_2\,ds.
\end{equation}
\end{lemma}
\begin{proof}
Let $u_1:=u^g_{\gamma_1}\in H^1(\Omega)$. From the variational equation,  we deduce 
\[
\int_\Omega\sigma \nabla u_1\cdot\nabla u_2\,dx+\int_\Gamma \gamma_1 u_1u_2\,ds=
\int_{\partial \Omega} g \Lambda(\gamma_2)g\, ds=\int_\Omega\sigma |\nabla u_2|^2\,dx+\int_\Gamma \gamma_2 u_2^2\,ds.
\]
Thus 
\[
\begin{aligned}
\int_\Omega\sigma |\nabla(u_1-u_2)|^2\,dx&+\int_\Gamma\gamma_1(u_1-u_2)^2\,ds \\
&=\int_\Omega\sigma |\nabla u_1|^2\,dx+ \int_\Gamma\gamma_1u_1^2\,ds
+\int_\Omega\sigma |\nabla u_2|^2\,dx+ \int_\Gamma\gamma_1u_2^2\,ds\\
&-2\int_\Omega\sigma |\nabla u_2|^2\,dx-2\int_\Gamma \gamma_2u^2_2\,ds\\
&=\int_{\partial \Omega} g \Lambda(\gamma_1)g\, ds
-\int_{\partial \Omega} g \Lambda(\gamma_2)g\, ds
+\int_\Gamma (\gamma_1-\gamma_2)u^2_2\,ds.
\end{aligned}
\]
Since the left-hand side is nonnegative, the first asserted inequality follows. 

Interchanging $\gamma_1$ and $\gamma_2$, we  get 
\[
\begin{aligned}
&\int_{\partial \Omega} g \Lambda(\gamma_2)g\, ds-\int_{\partial \Omega} g \Lambda(\gamma_1)g\, ds\\
&=\int_\Omega\sigma |\nabla(u_2-u_1)|^2\,dx+\int_\Gamma\gamma_2(u_2-u_1)^2\,ds
-\int_\Gamma (\gamma_2-\gamma_1)u^2_1\,ds\\
&=\int_\Omega\sigma |\nabla(u_2-u_1)|^2\,dx+\int_\Gamma\left(\gamma_2u^2_2-2\gamma_2u_1u_2+\gamma_1u^2_1\right)\,ds\\
&=\int_\Omega\sigma |\nabla(u_2-u_1)|^2\,dx+\int_\Gamma \gamma_1\left(u_1-\frac{\gamma_2}{\gamma_1}u_2 \right)^2\,ds+\int_\Gamma\left(\gamma_2-\frac{\gamma^2_2}{\gamma_1} \right)u^2_2\,ds.
\end{aligned}
\]
Since the first two integrals on the right-hand side are non negative, the second asserted inequality follows. 
\end{proof}

Note that we call Lemma \ref{mono} a monotonicity estimate  because of the following  corollary:
\begin{corollary}[Monotonicity]
For two Robin parameters $\gamma_1, \gamma_2 \in L^\infty_+(\Gamma)$
\begin{equation}
\gamma_1\leq \gamma_2 \quad \text{implies }\quad \Lambda(\gamma_1)\geq \Lambda(\gamma_2)\quad \text{in the sense of quadratic forms}.
\end{equation}
\end{corollary}
Let us stress, however, that Lemma \ref{mono} holds for any $\gamma_1,\gamma_2\in L^\infty_+(\Omega)$ and does not require $\gamma_1\leq \gamma_2$ or $\gamma_1\geq \gamma_2$. 

The existence of localized potentials \cite{gebauer2008localized} follows from the Runge approximation property as in \cite[Corollary~3.5]{harrach2018fractional}:
\begin{lemma}[Localized potentials]\label{lemma:locpot}
Let $\gamma\in L^\infty_+(\Gamma)$, and let  $M\subseteq \Gamma$ be a subset with positive boundary measure. 
Then there exists a sequence $(g_n)_{n\in\N}\subset L^2(\partial\Omega)$ such that the corresponding solutions $u^{(g_n)}$ of \eqref{transm} fulfill
\begin{equation*}
\lim_{n\to \infty}\int_{M} |u^{(g_n)}|^2\,ds=\infty \quad \text{ and }\quad
\lim_{n\to \infty}\int_{\Gamma\setminus M} |u^{(g_n)}|^2\,ds=0.
\end{equation*}
\end{lemma}
\begin{proof}
Using the Runge approximation property in Theorem~\ref{thm:runge} we find a sequence $\tilde g_n\in L^2(\partial \Omega)$
so that the corresponding solutions $u^{(\tilde g_n)}$ fulfill
\[
u^{(\tilde g_n)}|_\Gamma \to \frac{\chi_M}{\int_M\,ds}\quad \text{ in } L^2(\Gamma).
\]
Hence
\begin{equation*}
\lim_{n\to \infty}\int_{M} |u^{(\tilde g_n)}|^2\,ds=1  \quad \text{ and }\quad
\lim_{n\to \infty}\int_{\Gamma\setminus M} |u^{(\tilde g_n)}|^2\,ds=0,
\end{equation*}
so that 
\[
g_n:=\frac{\tilde g_n}{\left(\int_{\Gamma\setminus M}\tilde u_n^2\,ds\right)^{1/4}}, 
\]
has the desired property
\begin{align*}
\lim_{n\to \infty}\int_{M} |u^{(g_n)}|^2\,ds &= \lim_{n\to \infty} \frac{\int_{M} |u^{(\tilde g_n)}|^2\,ds}{\left(\int_{\Gamma\setminus M}
|u^{(\tilde g_n)}|^2\,ds\right)^{1/2}}=\infty,\\
\lim_{n\to \infty}\int_{\Gamma\setminus M} |u^{(g_n)}|^2\,ds &= \lim_{n\to \infty} \left(\int_{\Gamma\setminus M}|u^{(\tilde g_n)}|^2\,ds\right)^{1/2} = 0.
\end{align*}
\end{proof}

Now,  we are ready  to proof Theorem \ref{stability}.
\begin{proof}[Proof of Theorem \ref{stability}]
Let $\mathcal{A}$ be a finite dimensional subspace of $L^\infty(\Gamma)$, $b>a>0$, and 
\[
\gamma_1,\gamma_2\in \mathcal{A}_{[a,b]}=\left\{\gamma\in \mathcal{A}:\quad   a\leq\gamma(x)\leq b \text{ for } x\in \Gamma\ \text{(a.e.)}\right\}.   
\]
For the ease of notation, we write in the following
\[
\norm{\gamma_1-\gamma_2}:=\norm{\gamma_1-\gamma_2}_{L^\infty(\Omega)} \quad \text{ and } \quad
\norm{g}:=\norm{g}_{L^2(\partial \Omega)}.
\]
Since $\Lambda(\gamma_1)$ and $\Lambda(\gamma_2)$ are self-adjoint, we have that
\begin{align*}
\lefteqn{\norm{\Lambda(\gamma_2)-\Lambda(\gamma_1)}_*}\\
&= 
  \sup_{\norm{g}=1} \left| \int_{\partial \Omega} g \left(\Lambda(\gamma_2)-\Lambda(\gamma_1)\right) g\, ds \right|\\
  &=  \sup_{\norm{g}=1} \max\left\{
\int_{\partial \Omega} g \left(\Lambda(\gamma_2)-\Lambda(\gamma_1)\right) g\, ds,
\int_{\partial \Omega} g \left(\Lambda(\gamma_1)-\Lambda(\gamma_2)\right) g\, ds
\right\}.
\end{align*}
Using the first inequality in the monotonicity relation \eqref{eqmono} in Lemma \ref{mono} in its original form, and
with $\gamma_1$ and $\gamma_2$ interchanged, we obtain for all $g\in L^2(\partial \Omega)$
\begin{align*}
\int_{\partial \Omega} g \left(\Lambda(\gamma_2)-\Lambda(\gamma_1)\right) g\, ds &\geq 
\int_{\Gamma} (\gamma_1-\gamma_2)|u_{\gamma_1}^{(g)}|^2,\\
\int_{\partial \Omega} g \left(\Lambda(\gamma_1)-\Lambda(\gamma_2)\right) g\, ds &\geq 
\int_{\Gamma}(\gamma_2-\gamma_1)|u_{\gamma_2}^{(g)}|^2,
\end{align*}
where $u_{\gamma_1}^{(g)},u_{\gamma_2}^{(g)}\in H^1(\Omega)$ denote the solutions of \eqref{transm} with Neumann data $g$ and Robin parameter $\gamma_1$ and $\gamma_2$, resp.
Hence, for $\gamma_1\neq \gamma_2$, we have
\[
\frac{\| \Lambda(\gamma_2)-\Lambda(\gamma_1) \|_*}{\| \gamma_1 -\gamma_2\|}\geq  
 \sup_{\| g \|=1}\Psi\left(g,\frac{\gamma_1-\gamma_2}{\| \gamma_1 -\gamma_2\|_{L^\infty(\Gamma)}},\gamma_1,\gamma_2\right),
\]
where (for $g\in L^2(\partial \Omega)$, $\zeta\in \mathcal{A}$, and $\kappa_1,\kappa_2\in \mathcal{A}_{[a,b]}$)
\begin{equation}\label{eq:Definition_Psi}
\Psi\left(g,\zeta,\kappa_1,\kappa_2\right):=
\max \left\{ \int_{\Gamma} \zeta |u_{\kappa_1}^{(g)}|^2 \,ds, \int_{\Gamma} (-\zeta) |u_{\kappa_2}^{(g)}|^2 \,ds  \right\}.
\end{equation}
Introduce the compact set
\begin{equation}\label{eq:Definition_CC}
\mathcal{C}=\left\{ \zeta\in \mathcal{A}:\quad \| \zeta \|_{L^\infty(\Gamma)}=1 \right\}.
\end{equation}
Then, we have
\begin{equation}\label{eq:stability_infsup}
\frac{\| \Lambda(\gamma_2)-\Lambda(\gamma_1) \|_*}{\| \gamma_1 -\gamma_2\|}\geq
\inf_{\substack{\zeta\in \mathcal{C}\\ \kappa_1,\kappa_2\in\mathcal{A}_{[a,b]}}} \sup_{\| g \|=1} \Psi(g,\zeta,\kappa_1,\kappa_2).
\end{equation}
The assertion of Theorem \ref{stability} follows if we can show that the right hand side of \eqref{eq:stability_infsup}
is positive. Since $\Psi$ is continuous, the function
\[
(\zeta,\kappa_1,\kappa_2)\mapsto \sup_{\| g \|=1} \Psi(g,\zeta,\kappa_1,\kappa_2)
\]
is semi-lower continuous, so that it attains its minimum on  the compact set 
$\mathcal{C}\times\mathcal{A}_{[a,b]}\times \mathcal{A}_{[a,b]}$.
Hence, to prove Theorem \ref{stability}, it suffices to show that
\[
\sup_{\| g \|=1} \Psi(g,\zeta,\kappa_1,\kappa_2)>0 \quad \text{ for all } (\zeta,\kappa_1,\kappa_2)\in \mathcal{C}\times\mathcal{A}_{[a,b]}\times \mathcal{A}_{[a,b]}.
\]

To show this, let $(\zeta,\kappa_1,\kappa_2)\in \mathcal{C}\times\mathcal{A}_{[a,b]}\times \mathcal{A}_{[a,b]}$.
Since $\norm{\zeta}_{L^\infty(\Gamma)}=1$, there exists a subset $M\subseteq \Gamma$ with positive boundary measure such that
either
\[
\text{(a)}\ \zeta(x)\geq \frac{1}{2}\text{ for all } x\in M,\quad \text{ or } \quad
\text{(b)}\ -\zeta(x)\geq \frac{1}{2}\text{ for all } x\in M.
\]
In case (a), we use the localized potentials sequence in lemma~\ref{lemma:locpot} to obtain a boundary current
$\hat g\in L^2(\partial \Omega)$ with 
\[
\int_M \left|u^{(\hat g)}_{\kappa_1}\right|^2\, ds \geq 2 \quad \text{ and } \quad \int_{\Gamma\setminus M} \left|u^{(\hat g)}_{\kappa_1}\right|^2\, ds \leq \frac{1}{2},
\]
so that (using again $\norm{\zeta}_{L^\infty(\Gamma)}=1$)
\[
\Psi\left(\hat g,\zeta,\kappa_1,\kappa_2\right)\geq \int_\Gamma \zeta \left|u^{(\hat g)}_{\kappa_1}\right|^2\, ds \geq \frac{1}{2} \int_M \left|u^{(\hat g)}_{\kappa_1}\right|^2\, ds
- \int_{\Gamma\setminus M} \left|u^{(\hat g)}_{\kappa_1}\right|^2\, ds \geq \frac{1}{2}.
\]
In case (b), we can analogously use a localized potentials sequence for $\kappa_2$, and 
find $\hat g\in L^2(\partial \Omega)$ with
\[
\Psi\left(\hat g,\zeta,\kappa_1,\kappa_2\right)\geq \int_\Gamma (-\zeta) \left|u^{(\hat g)}_{\kappa_2}\right|^2\, ds \geq \frac{1}{2} \int_M \left|u^{(\hat g)}_{\kappa_2}\right|^2\, ds
- \int_{\Gamma\setminus M} \left|u^{(\hat g)}_{\kappa_2}\right|^2\, ds \geq \frac{1}{2}.
\]
Hence, in both cases, 
\[
\sup_{\| g \|=1} \Psi(g,\zeta,\kappa_1,\kappa_2)\geq  \Psi\left(\frac{\hat g}{\norm{\hat g}} ,\zeta,\kappa_1,\kappa_2\right)=
\frac{1}{\norm{\hat g}^2} \Psi(\hat g,\zeta,\kappa_1,\kappa_2)>0,
\]
so that Theorem \ref{stability} is proven.
\end{proof}


\section{Quantitative estimate of the Lipschitz stability constant}\label{section:quantitative}

Theorem \ref{stability} shows that there exists a Lipschitz stability constant $C>0$ with
\[
\| \gamma_1-\gamma_2 \|_{L^{\infty}(\Gamma)}\leq   C\| \Lambda(\gamma_1)-\Lambda(\gamma_2) \|_{*}\quad \text{ for all } \gamma_1,\gamma_2\in \mathcal{A}_{[a,b]},
\]
where the constant $C>0$ depends on the a-priori data only, i.e., on the sets $\Omega_1$ and $\Omega$, the background conductivity $\sigma$,
the finite dimensional subspace $\mathcal A\subset L^\infty(\Omega)$, and the a-priori bounds $b>a>0$.

Similar results on the existence of Lipschitz stability constants are known for several inverse coefficient problems, cf.\ the extensive reference list in the introduction. Lipschitz stability is often proven by constructive arguments such as quantitative unique continuation, and for some applications this allows 
to quantify the asymptotic behavior of the stability constants (cf., e.g., \cite[Corollary 2.5]{sincich2007lipschitz}).
However, to the knowledge of the authors, it is a long-standing unsolved problem how to explicitly calculate the stability constant from the a-priori data.

For practical purposes it is highly desirable to know the Lipschitz stability constant $C>0$ 
for a given setting. Inverse coefficient problems are usually extremely ill-posed, and only become well-posed (e.g., in the sense of Lipschitz stability as considered herein), when the unknown coefficient can a-priori be restricted to belong to a compact subset 
(e.g.\ the set of piecewise constant functions on a given resolution with a-priori known upper and lower bounds). Choosing a finer resolution leads to a less stable,  more noise-sensitive, reconstruction problem. Hence, a quantitative evaluation of the Lipschitz stability constant might help in 
estimating what resolution is practically feasible, cf.~\cite{harrach2015resolution} for results on a related problem concerning inclusion detection in EIT.

In this section, we show that monotonicity arguments allow a quantitative estimation of the Lipschitz stability constant by a relatively simple
implementation that requires solving a finite number of well-posed PDEs. To the knowledge of the authors, this is the 
first result on quantifying the Lipschitz stability constants for given a-priori data.

For the ease of presentation, we restrict ourself to the case that $\mathcal A$ consists of piecewise constant functions
on a given partition $\Gamma=\bigcup_{m=1}^M \Gamma_m$ into finitely many subsetes $\Gamma_m\subseteq \Gamma$ with positive boundary measure, i.e.,
\[
\mathcal A:=\{ \gamma=\sum_{m=1}^M \gamma_m \chi_{\Gamma_m},\ \text{ with } \gamma_1,\ldots,\gamma_M\in \R\}\subset L^\infty(\Gamma),
\]
and, for $b>a>0$, $\mathcal A_{[a,b]}$ is the subset of those $\gamma\in \mathcal A$ with $a\leq \gamma_m\leq b$ for all $m=1,\ldots,M$.
The authors believe that the following approach can also be extended to other finite-dimensional subspaces $\mathcal A\subseteq L^\infty(\Omega)$.

For our quantitative estimate of the Lipschitz stability constant, we require a finite number of localized potentials,
and show how to compute them.

\begin{lemma}\label{lemma:gkm}
For $m=1,\ldots,M$, and $k=1,\ldots,K$ with $K:=[4(\frac{b}{a}-1)]+1$, we define the Robin coefficient functions
\[
\gamma^{(km)}\in L^\infty_+(\Gamma) \quad \text{ by setting }
\quad \gamma^{(km)}:
=\left\{ \begin{array}{l l} (k+5)\frac{a}{4} & \text{ on $\Gamma_m$,}\\ \frac{a}{2} & \text{ else.}\end{array}\right.
\]

\begin{enumerate}
\item[(a)] There exist boundary data $g^{(km)}\in L^2(\partial \Omega)$ so that the corresponding solutions
$u_{km}\in H^1(\Omega)$ of \eqref{transm} with $\gamma=\gamma^{(km)}$ and $g=g^{(km)}$ fulfill
\begin{equation}\label{eq:condition_gkm}
\frac{1}{2} \int_{\Gamma_m}  |u_{km}|^2\,ds - \left( \frac{2b}{a}-1\right) \int_{\Gamma\setminus \Gamma_m}   |u_{km}|^2\,ds\geq 1
\end{equation}
\item[(b)] $g^{(km)}\in L^2(\partial \Omega)$ can be calculated by solving a finite number of well-posed partial differential equations.
\end{enumerate}
\end{lemma}
\begin{proof}
Note that $\gamma^{(km)}\in \mathcal A\cap L^\infty_+(\Gamma)$ but $\gamma^{(km)}\not\in \mathcal A_{[a,b]}$.
(a) immediately follows from the localized potentials result in lemma~\ref{lemma:locpot}. To prove (b), we use a similar approach as in the construction of localized
potentials in \cite[Thm.~2.10]{gebauer2008localized}. 
For $m\in \{1,\ldots,M\}$, and $k\in \{1,\ldots,K\}$, we introduce as in the proof of theorem~\ref{thm:runge}
\[
A:  L^2(\Gamma)\to  L^2(\partial\Omega), \quad f\mapsto A f:= v_{|\partial\Omega}, 
\]
where $v\in H^1(\Omega)$  solves 
\begin{equation*}
\int_\Omega\sigma\nabla v\cdot\nabla w\,dx+\int_\Gamma \gamma^{(km)} v w\,ds=\int_{\Gamma}f w\,ds\quad \text{ for all } w\in H^1(\Omega),
\end{equation*}
which is the variational form of the PDE
\begin{equation}
\label{eq:gkm_PDE1}
\left\{
\begin{aligned}
-\text{div}(\sigma \nabla v) & = 0\quad \text{ in }\Omega_1\cup \Omega_2,
\\
 \sigma\partial_{\nu} v&= 0\quad \text{ on  }\partial \Omega,\\
 \llbracket v\rrbracket&=  0 \quad\text{ on }  \Gamma,\\
 \llbracket \sigma\partial_{\nu} v\rrbracket &= \gamma^{(km)} v - f\quad\text{ on }  \Gamma,
\end{aligned}
\right.
\end{equation}

We have shown in the proof of theorem~\ref{thm:runge} that the adjoint of $A$ is given by
\[
A^*:\ L^2(\partial \Omega)\to L^2(\Gamma),\quad A^* g=u|_{\Gamma},
\]
where $u$ solves \eqref{transm} with $\gamma=\gamma^{(km)}$, and that $A^*$ has dense range.

Solving the ill-posed linear equation 
\[
A^* g \stackrel{!}{=} 4 \chi_{\Gamma_m}
\]
with the conjugate gradient (CGNE) method (cf., e.g., \cite[Sect.~7]{engl1996regularization}, or the recent book \cite[III.15]{hanke2017taste}) requires an application of $A$ and $A^*$ in each iteration step, i.e., one well-posed PDE solution of \eqref{transm} and \eqref{eq:gkm_PDE1}, each. Since $4 \chi_{\Gamma_m}\in \overline{\mathcal{R}(A^*)}$, the CGNE method is known to yield a (possibly unbounded) sequence of iterates $(g_n)_{n\in \N}\subset L^2(\partial \Omega)$ for which the residua converge to zero, i.e.,
\[
A^* g_n\to 4 \chi_{\Gamma_m}.
\]
Hence, the solutions $u^{(n)}\in H^1(\Omega)$ of \eqref{transm} with $\gamma=\gamma^{(km)}$ and $g=g_n$
fulfill
\[
\frac{1}{2} \int_{\Gamma_m}  |u^{(n)}|^2\,ds - \left( \frac{2b}{a}-1\right) \int_{\Gamma\setminus \Gamma_m}   |u^{(n)}|^2\,ds\to 2,
\]
so that after finitely many iteration steps, \eqref{eq:condition_gkm} is fulfilled.
\end{proof}

We can now formulate our quantitative estimate of the Lipschitz stability constant.
For the following theorem note that the functions $g^{(km)}$ (defined in lemma~\ref{lemma:gkm}) depend only
on the a priori data (i.e., on the sets $\Omega_1$ and $\Omega$, the background conductivity $\sigma$,
the finite dimensional subspace of piecewise constant functions $\mathcal A\subset L^\infty(\Omega)$, and the a-priori bounds $b>a>0$),
and that $g^{(km)}$ can be explicitly calculated from the a priori data by solving a finite number of well-posed PDEs.

\begin{theorem}\label{thm:quanitative_stability}
Let $g^{(km)}\in L^2(\partial \Omega)$ be defined as in lemma~\ref{lemma:gkm}. Set
\[
G:=\max\{ \norm{ g^{(km)}}^2:\ k=1,\ldots,K,\ m=1,\ldots,M \}.
\]
Then
\begin{equation}\label{eq:quanitative_stability}
\| \gamma_1-\gamma_2 \|_{L^{\infty}(\Gamma)}\leq   \frac{1}{G}\| \Lambda(\gamma_1)-\Lambda(\gamma_2) \|_{*}
\quad \text{ for all } \quad \gamma_1,\gamma_2\in \mathcal{A}_{[a,b]}. 
\end{equation}
\end{theorem}
\begin{proof}
We will first formulate an useful monotonicity-based inequality that allows us to modify the Robin coefficient in certain integral energy expressions (in (a)).
We then use it (in (b)) to estimate energy expressions for arbitrary Robin coefficients $\gamma\in \mathcal A_{[a,b]}$ by similar expressions that only involve the finitely many special Robin coefficients $\gamma^{(km)}$ defined in lemma \ref{lemma:gkm}.
From this we can then prove theorem \ref{thm:quanitative_stability}.

\begin{enumerate}
\item[(a)]
For all $g\in L^2(\partial \Omega)$, $\gamma\in L^\infty_+(\Gamma)$ and $\delta\in L^\infty(\Gamma)$ with $\gamma+\delta\in L_+^\infty(\Gamma)$, 
\begin{equation}
\label{eqmono_reformulated}
\int_\Gamma \delta |u_\gamma^g|^2\,ds\geq \int_\Gamma \delta |u_{\gamma+\delta}^g|^2\,ds.
\end{equation}
This follows from the first inequality in lemma~\ref{mono} with $\gamma_2:=\gamma+\delta$ and $\gamma_1:=\gamma$, and from using the same inequality again with interchanged roles of $\gamma_1$ and $\gamma_2$.
\item[(b)] Let $m\in \{1,\ldots,M\}$ and $\gamma\in \mathcal{A}_{[a,b]}$. 
Since $K$ (defined in lemma~\ref{lemma:gkm}) fulfills $b< (K+4)\frac{a}{4}$, 
there exists $k\in \{1,\ldots,K\}$ so that $\gamma_m:=\gamma|_{\Gamma_m}\in [a,b]$ fulfills
\[
(k+3)\frac{a}{4}\leq \gamma_m< (k+4)\frac{a}{4}.
\]

Using \eqref{eqmono_reformulated}, $\frac{a}{4}<(k+5)\frac{a}{4}- \gamma_m\leq \frac{a}{2}$, and $-\frac{a}{2}\geq \frac{a}{2}-\gamma\geq \frac{a}{2}-b$, we obtain (with $\gamma^{(km)}$, $g^{(km)}$, and $u_{km}$ defined as in lemma~\ref{lemma:gkm})
\begin{align*}
\lefteqn{ \int_{\Gamma_m} |u_\gamma^{g^{(km)}}|^2\,ds - \int_{\Gamma\setminus \Gamma_m} |u_\gamma^{g^{(km)}}|^2\,ds}\\
& = \frac{2}{a} \left(  \int_{\Gamma_m} \frac{a}{2} |u_\gamma^{g^{(km)}}|^2\,ds - \int_{\Gamma\setminus \Gamma_m}  \frac{a}{2} |u_\gamma^{g^{(km)}}|^2\,ds\right)\\
& \geq \frac{2}{a} \left(  \int_{\Gamma_m} ((k+5)\frac{a}{4}- \gamma_m) |u_\gamma^{g^{(km)}}|^2\,ds + \int_{\Gamma\setminus \Gamma_m}  \left(\frac{a}{2} - \gamma\right) |u_\gamma^{g^{(km)}}|^2\,ds\right)\\
& = \frac{2}{a}\int_\Gamma (\gamma^{(km)} - \gamma)  |u_\gamma^{g^{(km)}}|^2\,ds
 \geq \frac{2}{a}\int_\Gamma (\gamma^{(km)} - \gamma)  |u_{km}|^2\,ds\\
& = \frac{2}{a} \left(  \int_{\Gamma_m} ((k+5)\frac{a}{4}- \gamma_m) |u_{km}|^2\,ds + \int_{\Gamma\setminus \Gamma_m}  \left(\frac{a}{2} - \gamma\right) |u_{km}|^2\,ds\right)\\
& \geq \frac{2}{a} \left(   \int_{\Gamma_m} \frac{a}{4} |u_{km}|^2\,ds + \int_{\Gamma\setminus \Gamma_m}  \left(\frac{a}{2} - b\right) |u_{km}|^2\,ds\right)\\
&\geq \frac{1}{2} \int_{\Gamma_m}  |u_{km}|^2\,ds - \left( \frac{2b}{a}-1\right)  \int_{\Gamma\setminus \Gamma_m}   |u_{km}|^2\,ds \geq 1.
\end{align*}
This shows that for all $\gamma\in \mathcal{A}_{[a,b]}$ and all $m\in \{1,\ldots,M\}$
\begin{align}
\label{eq:energy_estimate_quant} \lefteqn{\sup_{\norm{g}=1} \left( \int_{\Gamma_m} |u_\gamma^{g}|^2\,ds - \int_{\Gamma\setminus \Gamma_m} |u_\gamma^{g}|^2\,ds \right)}\\
\nonumber & = \sup_{0\neq g\in L^2(\partial \Omega)} \frac{1}{\norm{g}^2} \left( \int_{\Gamma_m} |u_\gamma^{g}|^2\,ds - \int_{\Gamma\setminus \Gamma_m} |u_\gamma^{g}|^2\,ds \right)\geq \frac{1}{G}.
\end{align}
\item[(c)] It remains to show that \eqref{eq:energy_estimate_quant} implies \eqref{eq:quanitative_stability}. 
With $\Psi$ and $\mathcal C$ defined in \eqref{eq:Definition_Psi} and \eqref{eq:Definition_CC} in the proof of Theorem \ref{stability},
it suffices to show that
\begin{equation}\label{eq:quanitative_stability_aux}
\sup_{\| g \|=1} \Psi(g,\zeta,\kappa_1,\kappa_2)\geq \frac{1}{G} \quad \text{ for all } (\zeta,\kappa_1,\kappa_2)\in \mathcal{C}\times\mathcal{A}_{[a,b]}\times \mathcal{A}_{[a,b]}.
\end{equation}

Since $\mathcal A$ contains only piecewise-constant functions, for every $\zeta\in \mathcal C$ there must exist 
a boundary part $\Gamma_m$ with either 
\[
\zeta|_{\Gamma_m}=1, \quad \text{ or } \quad \zeta|_{\Gamma_m}=-1,
\]
and in both cases $-1\leq \zeta|_{\Gamma\setminus \Gamma_m}\leq 1$. Hence, using \eqref{eq:Definition_Psi} and \eqref{eq:energy_estimate_quant}, we obtain for the case
$\zeta|_{\Gamma_m}=1$:
\begin{align*}
\sup_{\| g \|=1} \Psi(g,\zeta,\kappa_1,\kappa_2)
&\geq \sup_{\| g \|=1} \int_{\Gamma} \zeta |u_{\kappa_1}^{(g)}|^2 \,ds\\
&\geq \sup_{\| g \|=1} \left( \int_{\Gamma_m} |u_{\kappa_1}^{g}|^2\,ds - \int_{\Gamma\setminus \Gamma_m} |u_{\kappa_1}^{g}|^2\,ds\right)
\geq \frac{1}{G},
\end{align*}
and for the case $\zeta|_{\Gamma_m}=-1$:
\begin{align*}
\sup_{\| g \|=1} \Psi(g,\zeta,\kappa_1,\kappa_2)
&\geq \sup_{\| g \|=1} \int_{\Gamma} (-\zeta) |u_{\kappa_2}^{(g)}|^2 \,ds\\
&\geq \sup_{\| g \|=1} \left( \int_{\Gamma_m} |u_{\kappa_2}^{g}|^2\,ds - \int_{\Gamma\setminus \Gamma_m} |u_{\kappa_2}^{g}|^2\,ds\right)
\geq \frac{1}{G},
\end{align*}
so that \eqref{eq:quanitative_stability_aux} and thus \eqref{eq:quanitative_stability} is proven.
\end{enumerate}
\end{proof}


\section{The minimization problem}
In order to solve numerically the inverse problem \eqref{invp}, we consider only Robin coefficients in the admissible set
\[
P_{ad}:=\left\{ \gamma\in  C(\Gamma):  0<c_0\leq\gamma(x)\leq c_1 \text{ for  } x\in \Gamma, \text{ where } c_0, c_1 \text{ are  constants}  \right\},
\]
and study the following minimization problem:
\begin{equation}
\label{minp}
 \min_{\gamma\in \mathcal{P}_{ad}} J(\gamma)=\frac{1}{2}\int_{\partial\Omega} (u(\gamma)-u_a)^2\,ds+\frac{\lambda}{2}\int_\Gamma \gamma^2\,ds,
 \end{equation}
 where $u(\gamma)$ solves the variational equation \eqref{eqv}, $u_a$  is the measured dada and $\lambda$ is  a regularization parameter  that must be chosen by the user.
 It is obvious  that the optimization  problem \eqref{minp} is equivalent to the inverse problem \eqref{invp} when the dada $u_a$ are exact  and the regularization parameter $\lambda=0$. When the data $u_a$ are contaminated  with  measurement errors,  the inverse problem  \eqref{invp} may not have a solution. However,  the optimization problem \eqref{minp}  has always  a solution, and  it may not     be equivalent  to the  solution of the inverse problem.
 \begin{theorem}
 The minimization problem \eqref{minp}  has at least  one solution.
 \end{theorem}
 \begin{proof}
 It is clear that $\inf J(\gamma)$ is finite over the  admissible set $P_{ad}$. Therefore there exists a minimizing sequence
 $\gamma_{n}$ such that 
\[
\lim_{n\to +\infty}J(\gamma_{n})=\inf_{\gamma\in P_{ad}} J(\gamma).
\] 
The  sequence $\gamma_{n}$ is bounded,  by the Banach-Alaoglu theorem, there exists a subsequence  still denoted  $\gamma_{n}$  such that 
\[
\lim_{n\to \infty}\gamma_{n}=  \gamma^* \quad \text{ weak-*  in  }  L^\infty(\Gamma). 
\] 
By definition $u(\gamma_{n})$ satisfies 
\begin{equation}
\label{varun}
\int_{\Omega}\sigma\nabla u(\gamma_{n})\cdot\nabla v\,dx+\int_\Gamma \gamma_n u(\gamma_n)v\,ds=\int_{\partial\Omega}gv\,ds \quad \forall  v\in H^1(\Omega).
\end{equation}
Taking $v=u(\gamma_{n})$ in \eqref{varun}, we obtain
\[
\int_{\Omega}\sigma|\nabla u(\gamma_{n})|^2\,dx+\int_\Gamma \gamma_n u(\gamma_n)^2\,ds=\int_{\partial\Omega}gu(\gamma_{n})\,ds.
\]
 Using the fact  that  $\| u\|^2:= \| \nabla u\|^2_{L^2(\Omega)}+\| u \|^2_{L^2(\Gamma)}$ is a  norm on $H^1(\Omega)$ equivalent to the natural norm (cf, \cite{MR3636710}) and  from the trace theorem, 
 we deduce  the existence of a constant $c>0$ such that
\[
\| u(\gamma_{n})\|_{H^1(\Omega)}\leq c \| g\|_{L^2(\partial\Omega)}.
\] 
This implies that $u(\gamma_{n})$ is uniformly bounded independent of $n$.  
Therefore there exists a subsequence  of $u(\gamma_{n})$ still denoted $u(\gamma_{n})$ such that 
 \[
 \lim_{n\to \infty}u(\gamma_{n})=u^*  \text{ weakly in } H^1(\Omega).
\]
 As the embedding of $H^1(\Omega)$ into  $L^2(\Omega)$ is compact, we also have
 \[
 \lim_{n\to \infty}u(\gamma_{n})=u^*  \text{  strongly in } L^2(\Omega).
 \]
 Now the strong convergence of $u(\gamma_n)$ in $L^2(\Gamma)$ and the weak-* convergence of $\gamma_n$ in $L^\infty(\Gamma)$ yield that
 \[
 \lim_{n\to\infty}\int_\Gamma \gamma_n u(\gamma_n)v\,ds= \int_\Gamma \gamma^* u^*v\,ds
 \]
Letting $n$ goes  to infinity in equation \eqref{varun}, we conclude that $u^*$ satisfies
\[
\int_{\Omega}\sigma\nabla u^*_{n}\cdot\nabla v\,dx+\int_\Gamma \gamma^*u^*v\,ds=\int_{\partial\Omega}gv\,ds,\quad \forall  v\in H^1(\Omega).
\]
Due to  the uniqueness of the weak limit, we get   $u(\gamma^*)=u^*$. This means that
\[
 \lim_{n\to \infty}u(\gamma_{n})=u(\gamma^*) \text{ weakly in } H^1(\Omega)  \text{ and   strongly in } L^2(\Omega).
 \]
Using the lower semi-continuity of the $L^2$-norm  yields
 \[
 J(\gamma^*)\leq \liminf_{n\to \infty}  J(\gamma_{n})=J(\gamma),
 \] 
which concludes the proof.
 \end{proof}
 In what follows we focus on the computation of the  derivative of  the function  $J$.
 \subsection{Derivative by the min-sup differentiability}
 We introduce the Lagrangian functional
 \[
 G(\gamma,\varphi,\psi)= J(\gamma,\varphi)+\int_\Omega \sigma\nabla\varphi\cdot\nabla\psi \,dx+\int_\Gamma \gamma\varphi\psi\,ds-\int_{\partial\Omega}g\psi\,ds, \quad \forall \varphi, \psi \in H^1(\Omega).
 \]
Then, it is easy to check that 
\[
J(\gamma,u(\Gamma))=\adjustlimits\min_{\varphi \in H^1(\Omega)}\sup_{\psi \in H^1(\Omega)}G(\gamma,\varphi,\psi),
\] 
since
\[
\sup_{\psi \in H^1(\Omega)}G(\Gamma,\varphi,\psi)= 
\left\{ \begin{array}{ccc} J(\gamma,u(\gamma)) &\text{ if } \varphi=u(\gamma),\\  +\infty  &\text{otherwise}. \end{array}\right.
\]
It is easily shown that the functional $G$  is convex continuous with respect to $\varphi$ and concave continuous with respect to $\psi$.
Therefore, according to Ekeland and Temam \cite{ET}, the functional $G$ has a saddle point $(u,v)$ if and only if $(u,v)$ solves the following system:
\[
\begin{aligned}
&\partial_{\psi}G(\gamma,u,v;\hat{\psi})=0, \\
&\partial_{\varphi}G(\gamma,u,v;\hat{\varphi})=0, 
\end{aligned}
\]
for all $\hat{\psi} \in H^1(\Omega)$ and $\hat{\varphi} \in H^1(\Omega)$. This yields that   $G$  has a saddle  point $(u,v)$,    where the  state  $u$  is the unique  solution of \eqref{eqv} and  the adjoint state 
 $v = v(\gamma)$ is the solution of the following adjoint  problem: 
 \begin{equation}\label{var-adj-neum}
 \int_{\Omega}\sigma\nabla v\cdot \nabla \hat{v}\,dx+\int_\Gamma \gamma v\hat{v}\,ds+\int_{\partial\Omega}(u-u_a)\hat{v}\,ds=0\quad \text{ for all } \hat{v}\in H^1(\Omega).
  \end{equation}
 Summarizing the above, we have  obtained
 \begin{theorem}
 The functionals  $J(\gamma,u(\gamma))$  is  given as
 \begin{equation}
J(\gamma,u(\gamma))=\adjustlimits\min_{\varphi \in H^1(\Omega)}\sup_{\psi \in H^1(\Omega)}G(\gamma,\varphi,\psi),
\end{equation}
The unique saddle point for $G$  is   given by  $(u,v)$.
\end{theorem}
\begin{theorem}\label{diffrobin}
 The functional $J$ is Gateaux  differentiable, and  its Gateaux derivative  at  $\gamma\in C(\Gamma)$ in the  direction 
 $\hat{\gamma}$ is given by 
\begin{equation}\label{DJ}
D_{\gamma}J(\gamma,u(\gamma);\hat{\gamma})=\int_{\Gamma}\hat{\gamma}\left(uv+\lambda\gamma\right) \,ds.
\end{equation}
\end{theorem}
\begin{proof}
Let $\gamma_{t}=\gamma+t\hat{\gamma}$, where $\hat{\gamma}\in C(\Gamma)$ and $t\in\mathbb{R}$ is sufficiently small parameter.    
Under hypotheses of Theorem \ref{CS}, we have
\[
D_{\gamma}J(\gamma,u(\gamma);\hat{\gamma})=\partial_{t}G(\gamma_{t},u(\gamma_t),v(\gamma_t))\Big|_{t=0},
\]
where
\[
 G(\gamma_t,\varphi,\psi)= J(\gamma,\varphi)+\int_\Omega \sigma\nabla\varphi\cdot\nabla\psi \,dx+\int_\Gamma \gamma_t\varphi\psi\,ds-\int_{\partial\Omega}g\psi\,ds, 
 \]
and 
\[
\partial_{t}G(\gamma_{t},u,v)\Big|_{t=0}=\int_{\Gamma}\hat{\gamma}\left(uv+\lambda\gamma\right)\,ds.
\]
From the above  equation  yields \eqref{DJ}. To end the proof, we should verify the four assumptions $(H_{1})-(H_{4})$   of Theorem \ref{CS} given in  the appendix . 
As in Theorem \ref{CS}, introduce the sets
\[
X(t) :=\left\{ x^t\in H^1(\Omega): \sup_{y\in H^1(\Omega)}G(t,x^t,y) = \adjustlimits\inf_{x\in H^1(\Omega)}\sup_{y\in H^1(\Omega)}G(t,x,y)\right\},
 \]
 \[
 Y(t)  :=\left\{ y^t\in H^1(\Omega): \inf_{x\in H^1(\Omega)}G(t,x,y^t)=\adjustlimits\sup_{y\in H^1(\Omega)}\inf_{x\in H^1(\Omega)}G(t,x,y)\right\},
 \]
 we obtain
 \[
\forall \;t\in [0,\varepsilon]\quad S(t)=X(t)\times Y(t)=\{u(\gamma_t),v(\gamma_t)\}\neq\varnothing,
\]
and assumption $(H_{1})$ is satisfied.\\

\noindent{\it Assumption $(H_{2})$: } 
The partial derivative $ \partial_{t} G(t,\varphi,\psi)$  exists everywhere in $[0,\varepsilon)$ and   the condition $(H_{2})$ is satisfied. \\
\noindent{\it Assumptions $(H_{3})$ and $(H_{4})$:} We show first the boundedness of  $(u(\gamma_t),v(\gamma_t))$.
 Let $w=u(\gamma_t)$ in the variational equation 
 \begin{equation}
 \label{varpt}
\int_\Omega\sigma\nabla u(\gamma_t)\cdot\nabla w\,dx+\int_\Gamma \gamma_t u(\gamma_t)w\,ds=\int_{\partial\Omega}gv\,ds \text{ for all } w\in  H^1(\Omega).
\end{equation}
 We get 
 \[
 \min(\sigma_1,\sigma_2)\|\nabla u(\gamma_t) \|^2_{L^2(\Omega)}+ c_0\| u(\gamma_t) \|^2_{L^2(\Gamma)}\leq \| g\|_{L^2(\partial \Omega)}\| u(\gamma_t) \|_{L^2(\partial \Omega)}.
 \]
 Using the fact  that  $\| u\|^2:= \| \nabla u\|^2_{L^2(\Omega)}+\| u \|^2_{L^2(\Gamma)}$ is a  norm on $H^1(\Omega)$ equivalent to the natural norm  and  from the trace theorem, there exists  $c>0$,  depending only on  $\Omega$  \ such that 
 \[ \| u(\gamma_t) \|_{H^1(\Omega)}\leq c \| g\|_{L^2(\partial \Omega)},\]
which yields
\[ \sup_{t\in [0,\varepsilon)}\|u(\gamma_t) \|_{H^1(\Omega)}\leq c\| g\|_{L^2(\partial \Omega)}.\]
We apply the same technique to the variational equation 
\begin{equation}
\label{varadjt}
 \int_{\Omega}\sigma\nabla v(\gamma_t)\cdot \nabla \hat{v}\,dx+\int_\Gamma \gamma_t v(\gamma_t)\hat{v}\,ds+\int_{\partial\Omega}(u(\gamma_t)-u_a)\hat{v}\,ds=0,
 \end{equation}
  for all  $\hat{v}\in H^1(\Omega)$, and we  are able to show that the function $v(\gamma_t)$ is bounded. 
The next step is to show the continuity with respect to $t$ of the vector $(u(\gamma_t),v(\gamma_t))$. Subtracting \eqref{varpt} at $t>0$ and   $t=0$ and  choosing  
$w=u(\gamma)-u(\gamma_t)$ yields   
 \begin{align*}
 &\int_{\Omega}\sigma|\nabla(u(\gamma)-u(\gamma_t))|^2\,dx +\int_{\Gamma}\gamma(u(\gamma)-u(\gamma_t))^2\, ds  
 = \int_{\Gamma}(\gamma_t-\gamma)u(\gamma_t)(u(\gamma)-u(\gamma_t))\\
& \leq \| \gamma_t-\gamma\|_{L^\infty(\Gamma)}\| u(\gamma_t) \|_{L^2(\Gamma)}\| u(\gamma)-u(\gamma_t)\|_{L^2(\Gamma)} .
\end{align*}
Furthermore due to the boundedness  of $u(\gamma_t)$, we obtain
 \[
  \| u(\gamma_t)-u(\gamma)\|_{H^1(\Omega)}\leq c\| \gamma_t-\gamma \|_{L^\infty(\Gamma)}.
 \]
 Due to the strong continuity  of $\gamma_t$ as a function of $t$, one deduces that  $u(\gamma_t)\rightarrow u(\gamma)$ in $H^1(\Omega)$ as $t \rightarrow 0$. 
 Concerning  the  continuity of $v(\gamma_t)$,  one may  show  from \eqref{varadjt} that $v(\gamma_t)\rightarrow v(\gamma)$ in $H^1(\Omega)$.
Finally in  view of the strong continuity of 
\[
(t,\varphi)\rightarrow \partial_{t}G(\gamma_t,\varphi,\psi),  \quad   (t,\psi)\rightarrow \partial_{t}G(\gamma_t,\varphi,\psi),   
\]
 assumptions $(H_{3})$ and $(H_{4})$ are verified. 
\end{proof}
\section{Implementation details and numerical examples}
In the following numerical examples,  the domain $\Omega$  under consideration is the unit disk centered at  the origin and  the boundary $\Gamma$ is  given by: 
\[
\Gamma=\left\{ (x_1,x_2)\in \R^2: \quad  x^2_1+x^2_2=0.5^2  \right\}.
\]
The domain $\Omega$ is discretized using a Delaunay triangular mesh.  A  standard finite element method with piecewise finite elements is applied to compute numerically the state and the adjoint state for our problem. 
 The exact  data $u_a$  are computed  synthetically by solving the direct problem \eqref{transm}. 
 In   the real-world, the data $u_a$  are experimentally acquired and thus always contaminated  by  errors.
In our numerical examples the simulated noise data  are generated using the following formula:
\[
\tilde{u}_a(x_1,x_2)= u_a(x_1,x_2)\left(1+\varepsilon\delta\right)\quad \text{ on } \partial\Omega,
\]
where $\delta$ is a normal distributed random variable and $\varepsilon$ indicates the level of noise.  For our examples, the random variable $\delta$ is realized using 
the matlab   function randn$()$.   The conductivity $\sigma$  is  taken to be 
\[
\sigma=2\chi_{\Omega_1}+\chi_{\Omega_2}.  
\]
We use the BFGS algorithm to solve the minimization problem \eqref{minp}.  This quasi-Newton method is well adapted  to such problem.   
\subsection{Numerical examples}
For the following numerical examples,  we include both reconstruction from a noiseless  and noisy data. 
The regularization parameter $\lambda$   is set to zero because it does not  seem to play an indispensable role in our numerical  experiments.

Figure \ref{geom1},   shows the monotonicity relation between the Robin coefficient  $\gamma$ and the Neumann-to-Dirichlet 
operator $\Lambda(\gamma)$. 
\subsubsection{Example 1}
In this example,  we  use three measurements corresponding to the following fluxes:
\[
g_1(\theta)=\cos(\theta), \quad g_2(\theta)=\sin(\theta)\quad \text{ and }\quad g_3(\theta)=\cos^2(\theta)-\sin^2(\theta) \text{ on }\partial\Omega. 
\] 
In this case the cost function $J$ takes the form:
\[
J(\gamma)=\sum_{k=1}^3\int_{\partial\Omega}(u^{g_k}_\gamma)-u^{g_k}_a)^2\,ds+\frac{\lambda}{2}\int_\Gamma \gamma^2\,ds,
\]
where  $u^{g_k}_\gamma$ is the solution to problem \eqref{transm}  with respect to the flux $g_k$ and  $u^{g_k}_a$ is the 
corresponding measurement data.

The exact Robin coefficient  to be reconstructed is given by
\[
 \gamma(\theta)=\exp\left(-\frac{1}{2}\cos(\theta)\right), \quad \theta \in [0,2\pi].
 \]
Since the  cost functional is non convex, then it might  have some local minima. Therefore the  accuracy  of the reconstruction of the Robin coefficients depends   on   the initial guess  $\gamma_0$ as depicted in figure \ref{geom2}.  The numerical solution represents reasonable approximation  and  it is stable with respect to a small
amount  of noise   as shown in  figures \ref{geom5}, \ref{geom8}.
Figures \ref{geom3}, \ref{geom4}, \ref{geom6}, \ref{geom7}, \ref{geom9}, \ref{geom10}, shown  the decrease of cost function $J$  and the $L^\infty$-norm of $DJ(\gamma)$ in the course of the optimization process.

\begin{figure}[ht]
 \begin{minipage}[b]{0.45\linewidth}
 \includegraphics[width=5.75cm, height=5.5cm]{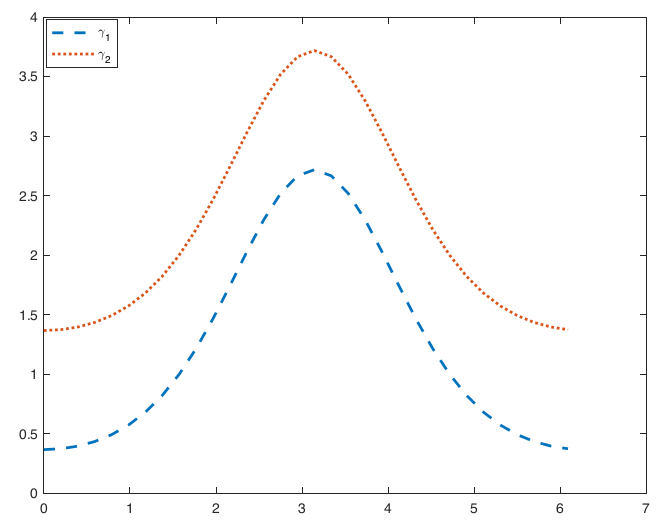}
 \centering{$(a)$}
 \end{minipage}
 \hfill
\begin{minipage}[b]{0.45\linewidth}
\includegraphics[width=5.75cm, height=5.5cm]{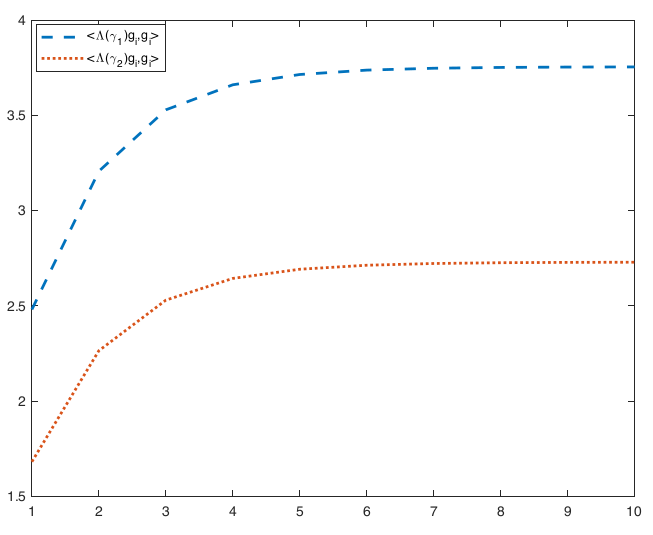}
 \centering{$(b)$}
\end{minipage}
 \caption{Monotonicity of the Neumann-to-Dirichlet operator $g_i\rightarrow \Lambda(\gamma)g_i$. In  figure  (a) $\gamma_1=\exp(-\cos(\theta))\leq \gamma_2=\exp(-\cos(\theta))+1, \theta\in[0,2\pi]$. In figure  (b),  $\langle \Lambda(\gamma_1)g_i,g_i\rangle\geq \langle \Lambda(\gamma_2)g_i,g_i\rangle$ with $g_i=\sin(i\theta), i=1,\ldots,10$.}
 \label{geom1}
 \end{figure}
 \begin{figure}[ht]
 \begin{minipage}[b]{0.45\linewidth}
 \includegraphics[width=5.75cm, height=5.5cm]{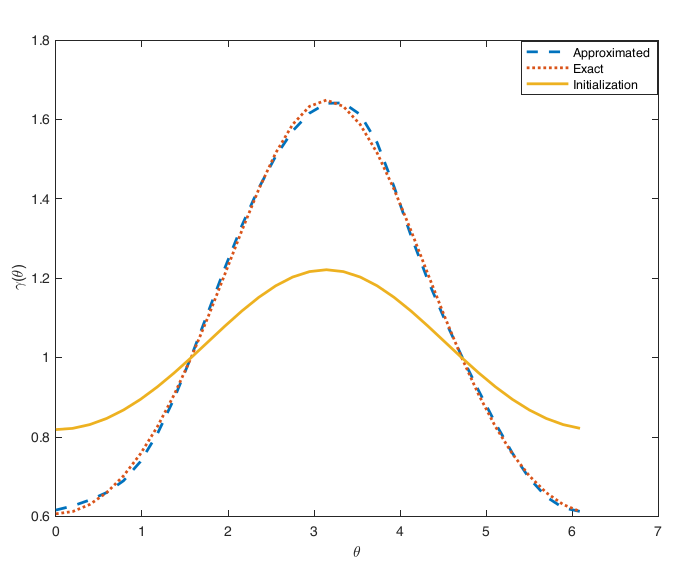}
 \centering{}
 \end{minipage}
 \hfill
\begin{minipage}[b]{0.45\linewidth}
\includegraphics[width=5.75cm, height=5.5cm]{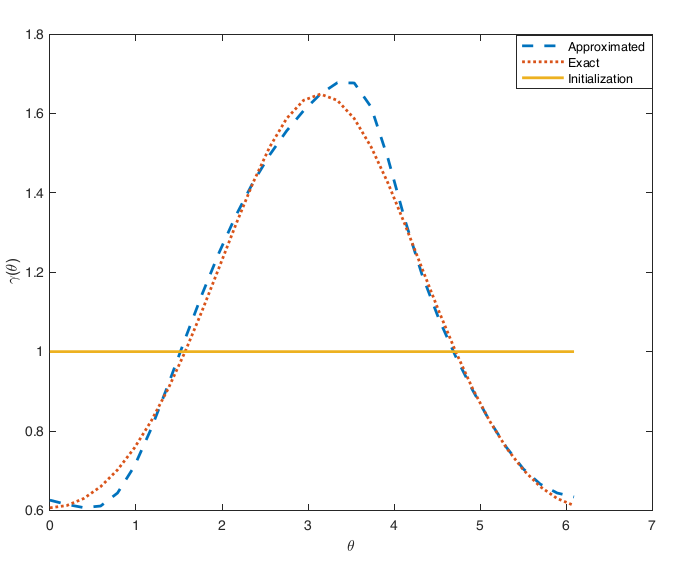}
 \centering{}
\end{minipage}
 \caption{Simulation results for Example1:  Reconstruction of the Robin coefficient for two different initialization $\gamma_{i}(\theta)=\exp(-0.2\cos(\theta))$ and $\gamma_i=1$, with level noise $\varepsilon=0$.}
  \label{geom2}
 \end{figure}
\begin{figure}[ht]
 \begin{minipage}[b]{0.45\linewidth}
 \includegraphics[width=5.75cm, height=5.5cm]{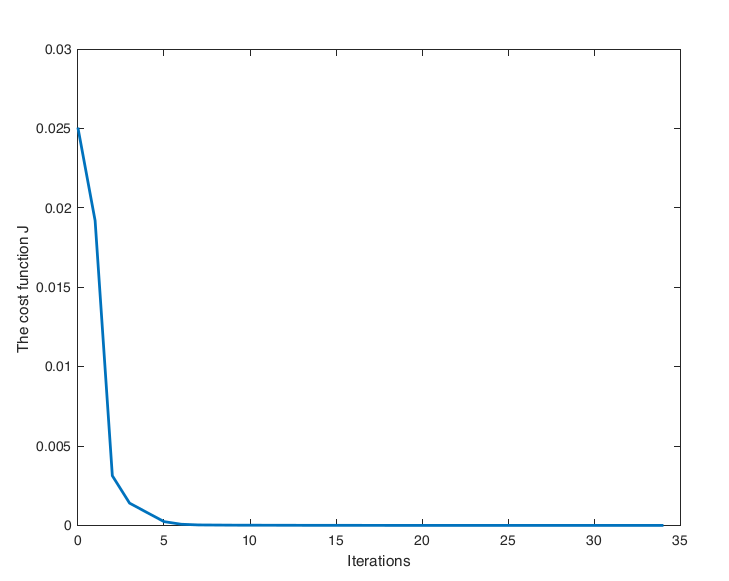}
 \centering{}
 \end{minipage}
 \hfill
\begin{minipage}[b]{0.45\linewidth}
\includegraphics[width=5.75cm, height=5.5cm]{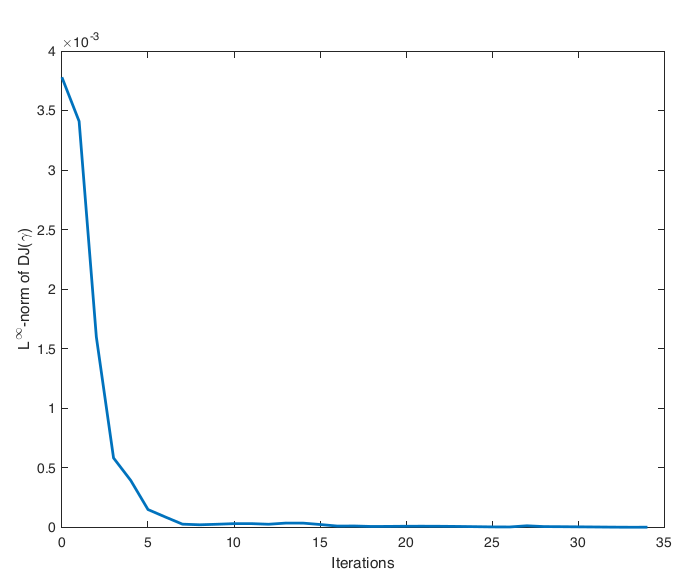}
 \centering{ }
\end{minipage}
 \caption{Simulation results for Example1: History of the cost function $J(\gamma)$ and the $L^\infty$-norm of $DJ(\gamma)$  in the case of the initialization  $\gamma_i=\exp(-0.2\cos(\theta))$ and  
 the level noise $\varepsilon=0$.}
  \label{geom3}
 \end{figure}
 \begin{figure}[ht]
 \begin{minipage}[b]{0.45\linewidth}
 \includegraphics[width=5.75cm, height=5.5cm]{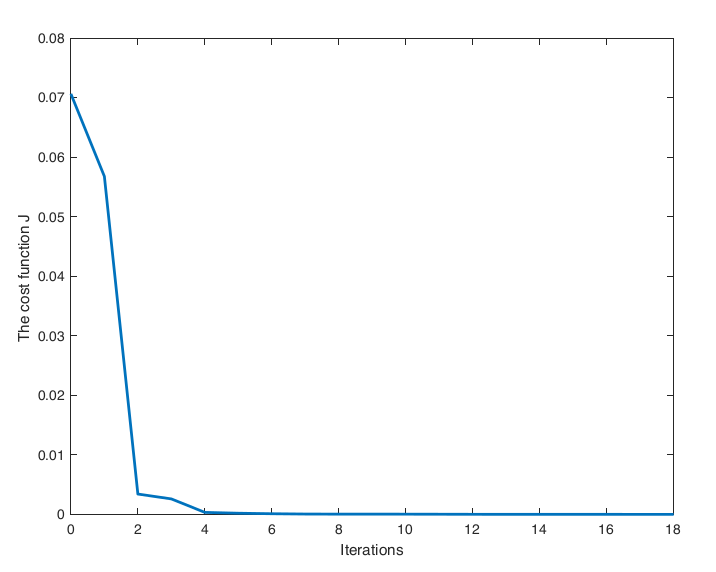}
 \centering{}
 \end{minipage}
 \hfill
\begin{minipage}[b]{0.45\linewidth}
\includegraphics[width=5.75cm, height=5.5cm]{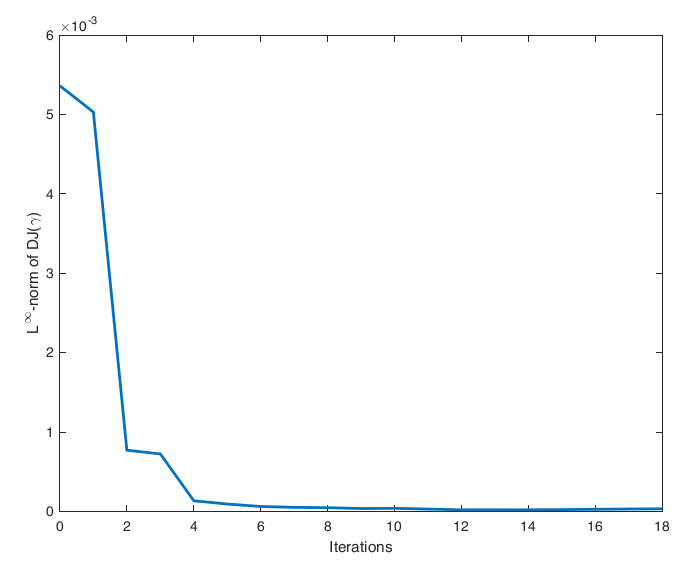}
 \centering{}
\end{minipage}
 \caption{Simulation results for Example1: History of the cost function $J(\gamma)$ and the $L^\infty$-norm of $DJ(\gamma)$  in the case of the initialization  $\gamma_i=1$ and  
 the level noise $\varepsilon=0$.}
  \label{geom4}
 \end{figure}
 \begin{figure}[ht]
 \begin{minipage}[b]{0.45\linewidth}
 \includegraphics[width=5.75cm, height=5.5cm]{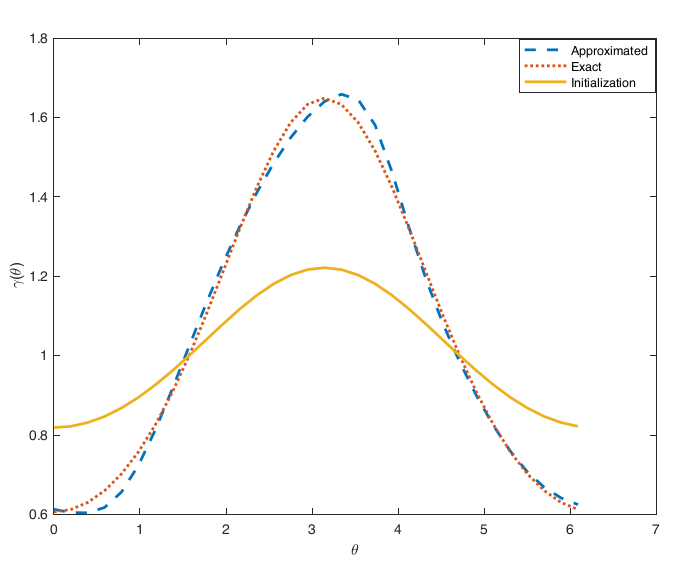}
 \centering{(c)}
 \end{minipage}
 \hfill
\begin{minipage}[b]{0.45\linewidth}
\includegraphics[width=5.75cm, height=5.5cm]{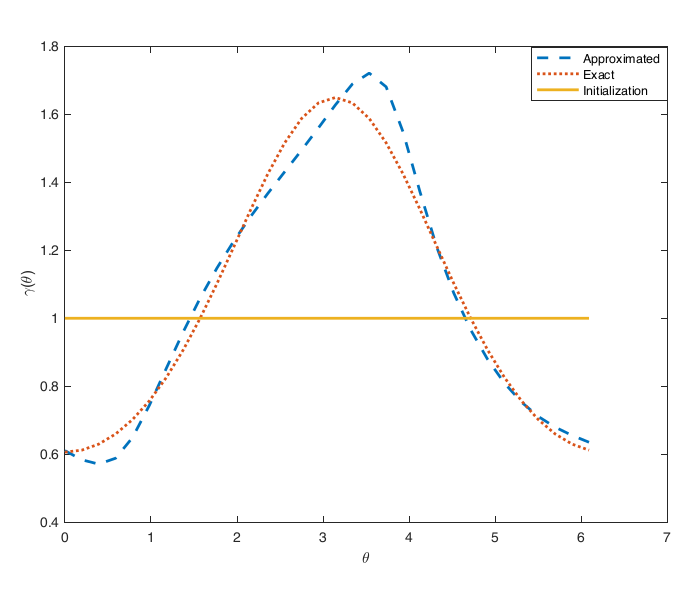}
 \centering{ (d)}
\end{minipage}
 \caption{Simulation results for Example1: Reconstruction of the Robin coefficient 
  for two different initialization $\gamma_{i}(\theta)=\exp(-0.2\cos(\theta))$ and $\gamma_i=1$, 
  with level noise $\varepsilon=0.01$ for (c),  and  $\varepsilon=0.03$ for (d).}
  \label{geom5}
 \end{figure}
  \begin{figure}[ht]
 \begin{minipage}[b]{0.45\linewidth}
 \includegraphics[width=5.75cm, height=5.5cm]{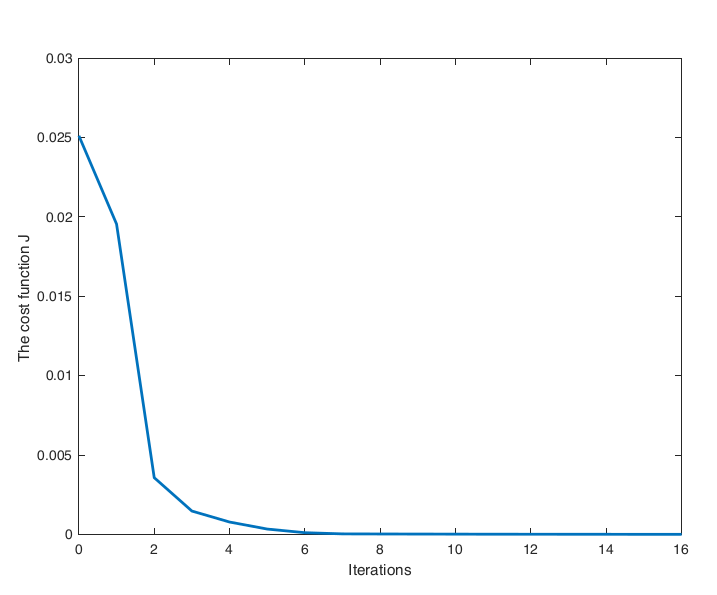}
 \centering{}
 \end{minipage}
 \hfill
\begin{minipage}[b]{0.45\linewidth}
\includegraphics[width=5.75cm, height=5.5cm]{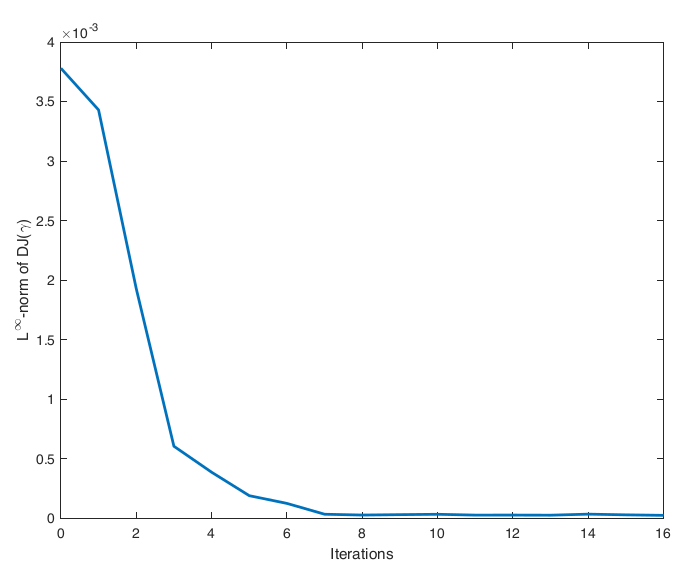}
 \centering{}
\end{minipage}
 \caption{Simulation results for Example 1: History of the objective function $J(\gamma)$ and the L$^\infty$-norm of $DJ(\gamma)$
 in the case of the initialization $\gamma_i=\exp(-0.2\cos(\theta))$ and level noise $\varepsilon=0.01$.}
  \label{geom6}
 \end{figure}
  \begin{figure}[ht]
 \begin{minipage}[b]{0.45\linewidth}
 \includegraphics[width=5.75cm, height=5.5cm]{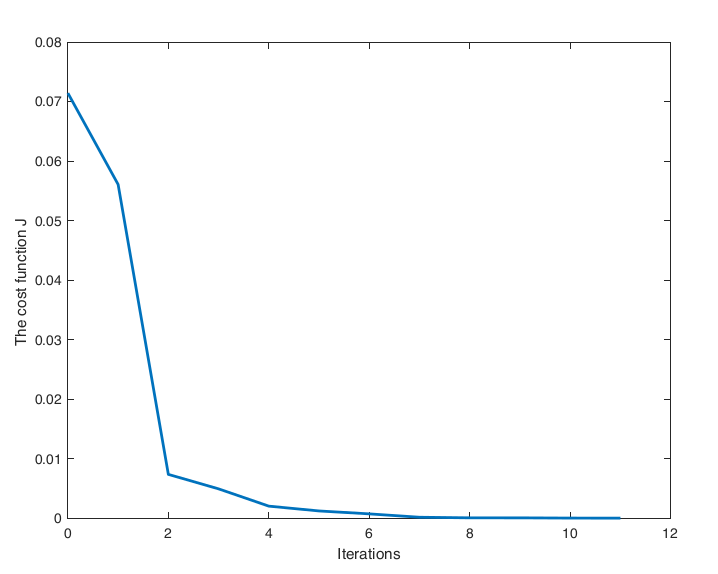}
 \centering{}
 \end{minipage}
 \hfill
\begin{minipage}[b]{0.45\linewidth}
\includegraphics[width=5.75cm, height=5.5cm]{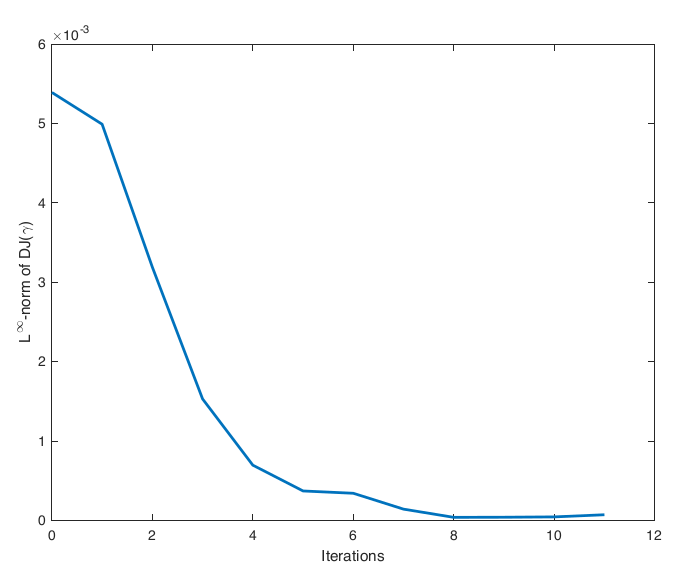}
 \centering{}
\end{minipage}
 \caption{Simulation results for Example 1: History of the objective function $J(\gamma)$ and the L$^\infty$-norm of $DJ(\gamma)$
 in the case of the initialization $\gamma_i=1$ and level noise $\varepsilon=0.03$.}
  \label{geom7}
 \end{figure} 
 \begin{figure}[ht]
 \begin{minipage}[b]{0.45\linewidth}
 \includegraphics[width=5.75cm, height=5.5cm]{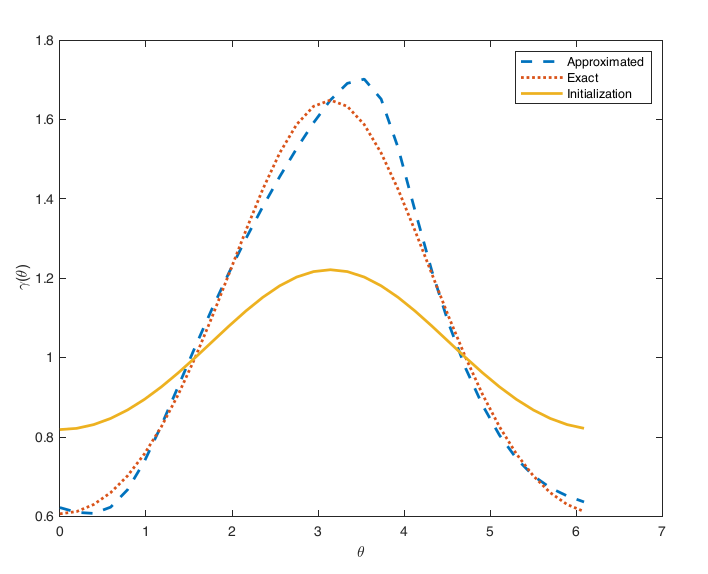}
 \centering{$(e)$}
 \end{minipage}
 \hfill
\begin{minipage}[b]{0.45\linewidth}
\includegraphics[width=5.75cm, height=5.5cm]{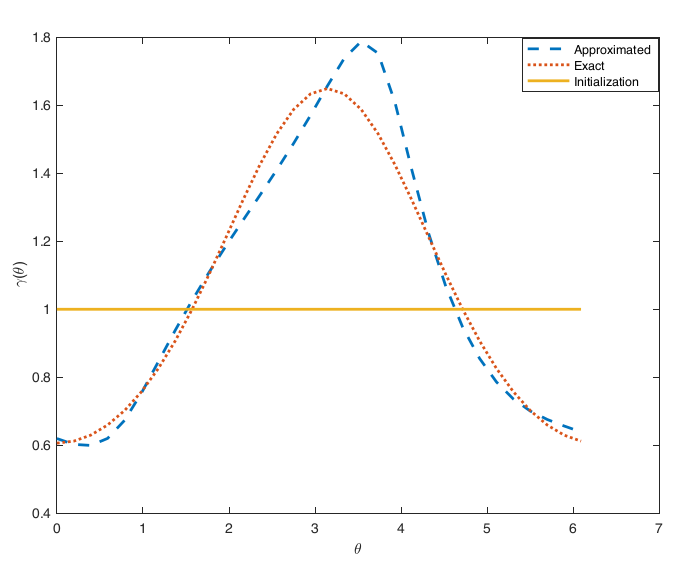}
 \centering{ (f)}
\end{minipage}
 \caption{Simulation results for Example 1: Reconstruction of the Robin coefficient for two different initialization $\gamma_{i}(\theta)=\exp(-0.2\cos(\theta))$ and $\gamma_i=1$,   with level noise $\varepsilon=0.05$ for (e), and  $\varepsilon=0.1$ for $(f)$.}
  \label{geom8}
\end{figure}
   \begin{figure}[ht]
 \begin{minipage}[b]{0.45\linewidth}
 \includegraphics[width=5.75cm, height=5.5cm]{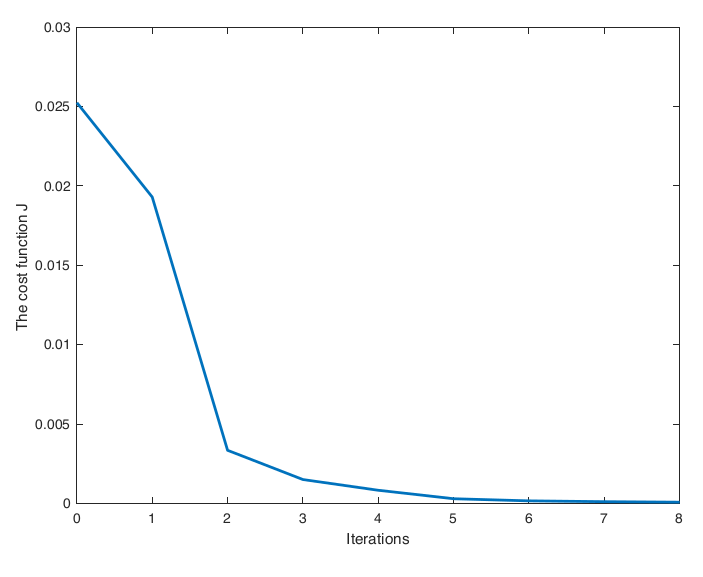}
 \centering{}
 \end{minipage}
 \hfill
\begin{minipage}[b]{0.45\linewidth}
\includegraphics[width=5.75cm, height=5.5cm]{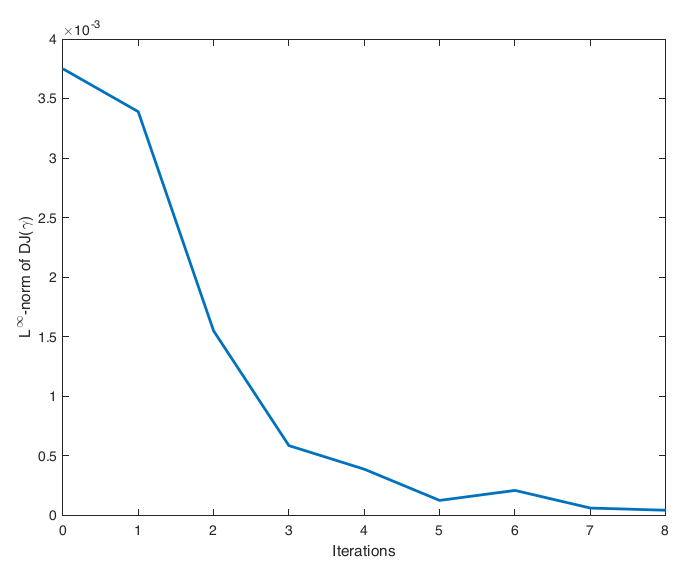}
 \centering{ }
\end{minipage}
 \caption{Simulation results for Example 1: History of the objective function $J(\gamma)$ and the L$^\infty$-norm of $DJ(\gamma)$
 in the case of the initialization $\gamma_i=\exp(-0.2\cos(\theta))$ and level noise $\varepsilon=0.05$.}
  \label{geom9} 
  \end{figure}
  \begin{figure}[ht]
 \begin{minipage}[b]{0.45\linewidth}
 \includegraphics[width=5.75cm, height=5.5cm]{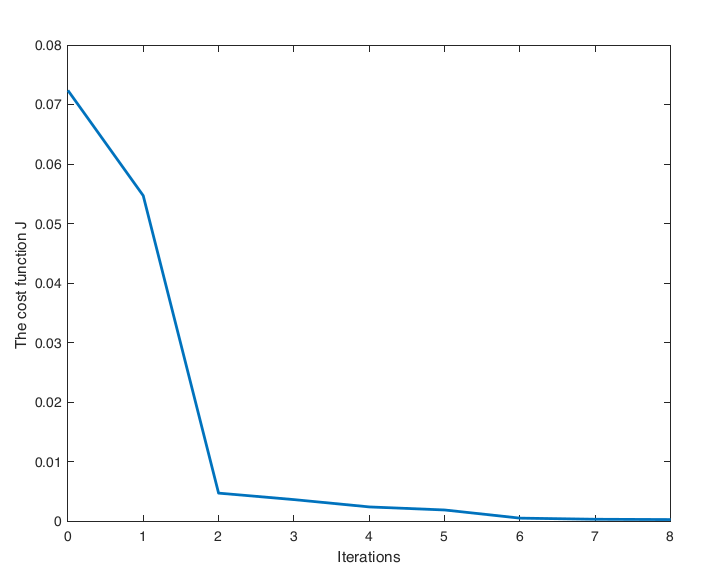}
 \centering{}
 \end{minipage}
 \hfill
\begin{minipage}[b]{0.45\linewidth}
\includegraphics[width=5.75cm, height=5.5cm]{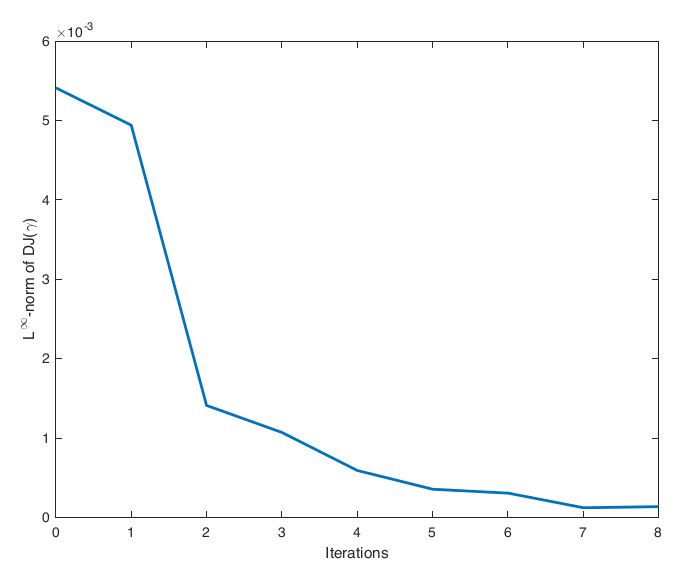}
 \centering{ }
\end{minipage}
 \caption{Simulation results for Example 1: History of the objective function $J(\gamma)$ and the L$^\infty$-norm of $DJ(\gamma)$
 in the case of the initialization $\gamma_i=1$ and level noise $\varepsilon=0.1$.}
  \label{geom10} 
  \end{figure} 
  \subsubsection{Example 2}
In the second example,  we consider  the   following three  fluxes 
\[
g_1(\theta)=5+\cos(\theta), \quad g_2(\theta)=1+\sin(\theta)\quad \text{ and }  \quad g_3(\theta)=3+\cos^2(\theta)-\sin^2(\theta) \text{ on }\partial\Omega. 
\] 
The exact Robin coefficient  to be reconstructed is given by
\[
 \gamma(\theta)=1+\cos(\theta)^2, \quad \theta \in [0,2\pi].
 \]
A  reasonable reconstruction of the Robin coefficient with noiseless and noisy data  can be seen in figure \ref{geom11}.
Figures \ref{geom12}, \ref{geom13}, shown  the decrease of cost function $J$  and the $L^\infty$-norm of $DJ(\gamma)$ in the course of the optimization process.
  \begin{figure}[ht]
 \begin{minipage}[b]{0.45\linewidth}
 \includegraphics[width=5.75cm, height=5.5cm]{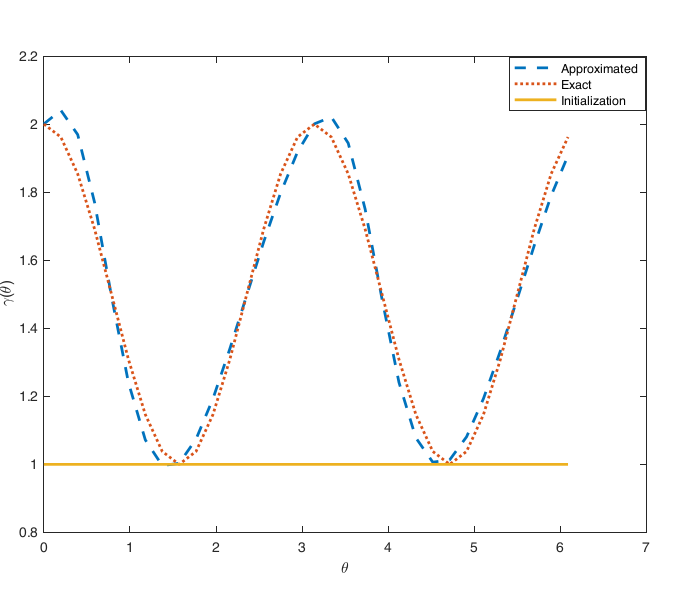}
 \centering{$(g)$}
 \end{minipage}
 \hfill
\begin{minipage}[b]{0.45\linewidth}
\includegraphics[width=5.75cm, height=5.5cm]{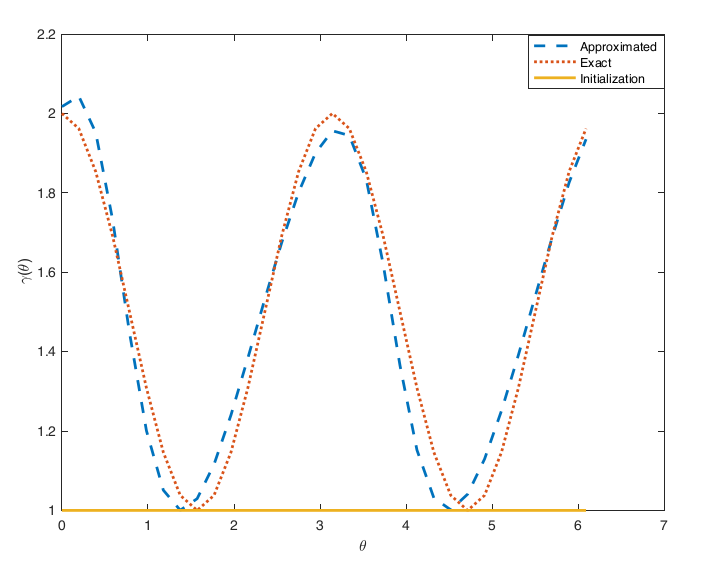}
 \centering{$(h)$}
\end{minipage}
 \caption{Simulation results for Example 2: Reconstruction of the Robin coefficient  with initialization  $\gamma_i=1$, and  level noise $\varepsilon=0.00$ for $(g)$,  and $\varepsilon=0.05$ for $(h)$.}
  \label{geom11} 
  \end{figure} 
  \begin{figure}[ht]
 \begin{minipage}[b]{0.45\linewidth}
 \includegraphics[width=5.75cm, height=5.5cm]{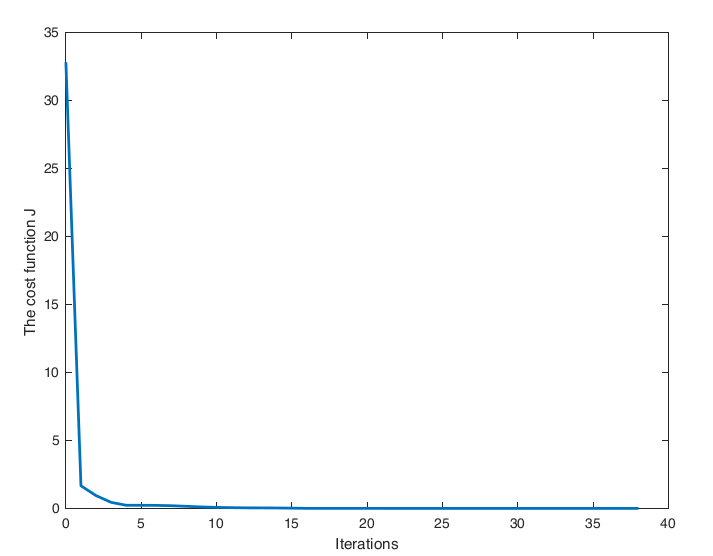}
 \centering{}
 \end{minipage}
 \hfill
\begin{minipage}[b]{0.45\linewidth}
\includegraphics[width=5.75cm, height=5.5cm]{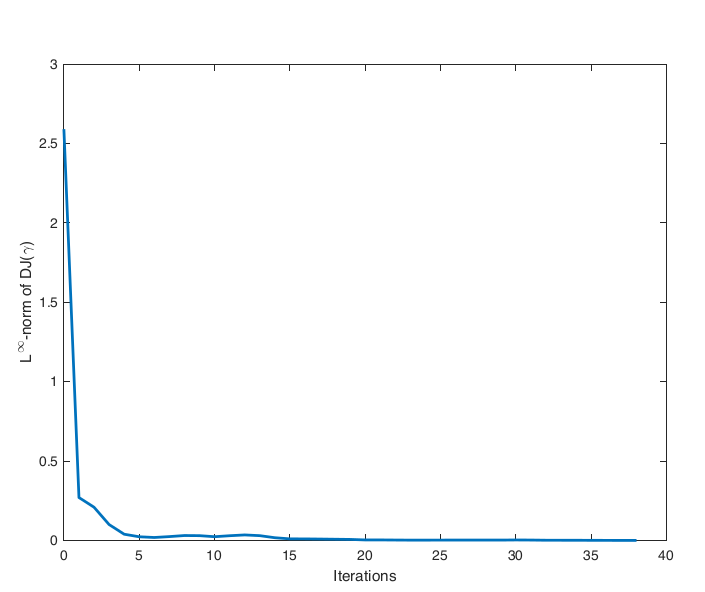}
 \centering{ }
\end{minipage}
 \caption{Simulation results for Example 2:  History of the objective function $J(\gamma)$ and the L$^\infty$-norm of $DJ(\gamma)$
 in the case of the initialization $\gamma_i=1$ and level noise $\varepsilon=0.0$.}
  \label{geom12} 
  \end{figure} 
    \begin{figure}[ht]
 \begin{minipage}[b]{0.45\linewidth}
 \includegraphics[width=5.75cm, height=5.5cm]{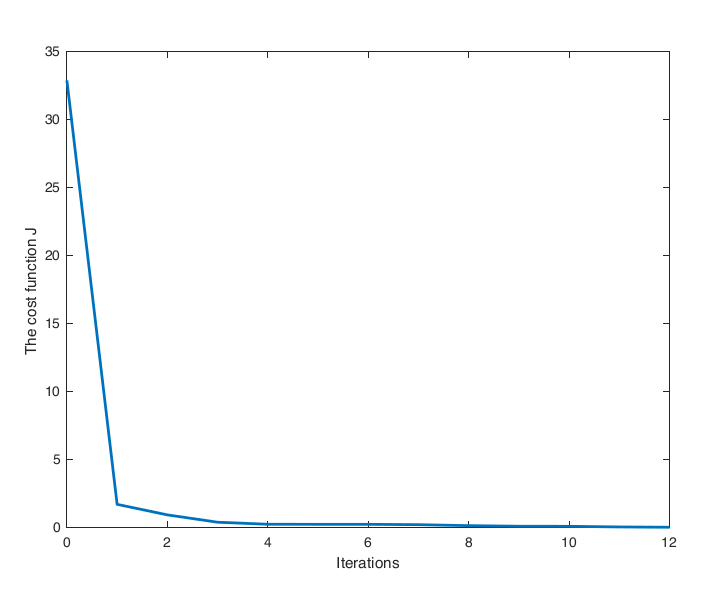}
 \centering{}
 \end{minipage}
 \hfill
\begin{minipage}[b]{0.45\linewidth}
\includegraphics[width=5.75cm, height=5.5cm]{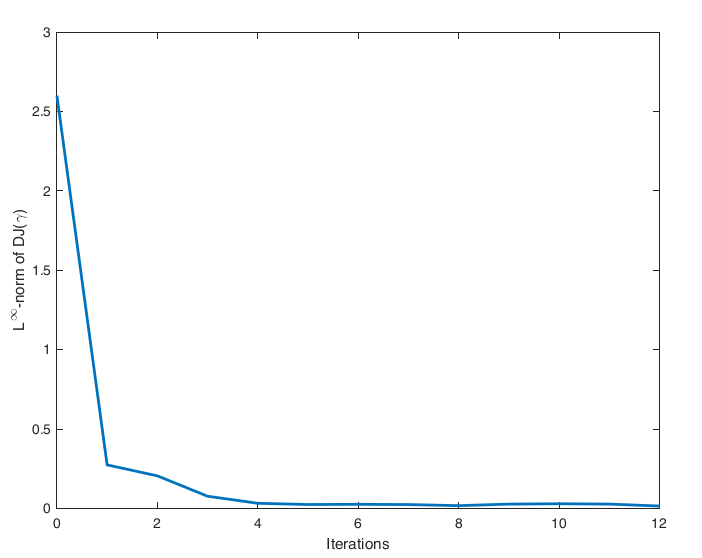}
 \centering{ }
\end{minipage}
 \caption{Simulation results for Example 2: History of the objective function $J(\gamma)$ and the L$^\infty$-norm of $DJ(\gamma)$
 in the case of the initialization $\gamma_i=1$ and level noise $\varepsilon=0.05$.}
  \label{geom13} 
  \end{figure} 
  \section{Conclusion}
We have studied theoretically and numerically  the inverse Robin transmission problem  from boundary measurements.
More precisely, we proved global  uniqueness and Lipschitz stability.  Our proofs rely on monotonicity and localized potentials arguments
and seem simpler than the commonly used technique of Carleman estimates or complex geometrical optics (CGO). Notably this is the first work to prove Lipschitz stability using monotonicity and localized potentials arguments, and the first work showing how to explicitly calculate the Lipschitz stability constant for a given setting.

The inverse problem is recast into a minimization of an output least-square formulation and a  technique based on the differentiability of the min-sup is suggested to drive the first order optimality condition without the derivative of the state solution. 
The BFGS method is employed  to solve the minimization problem, and  numerical results  for two-dimensional problem 
in the case of noiseless and noisy  data  are presented to illustrate the efficiency of the proposed algorithm.  
  \section{Appendix}
\subsection{An abstract differentiability result}\label{app1}
In this section we give an abstract result for differentiating Lagrangian functionals with respect to a parameter. This result is used to prove Theorem \ref{DJ}. We first introduce some notations.  Consider the functional
\begin{equation}
G: [0,\varepsilon]\times X\times Y\rightarrow \R
\end{equation}
for some $\varepsilon>0$ and the Banach spaces $X$ and $Y$. For each
$t\in [0,\varepsilon]$,  define
\begin{equation}\label{s2}
g(t)=\adjustlimits\inf_{x\in X}\sup_{y\in Y}G(t,x,y),\qquad  h(t)=\adjustlimits \sup_{y\in Y}\inf_{x\in X}G(t,x,y),
\end{equation}
and the associated sets
\begin{align}
 X(t) &=\left\{ x^t\in X: \sup_{y\in Y}G(t,x^t,y)=g(t)\right\},\\
 Y(t) &=\left\{ y^t\in Y: \inf_{x\in X}G(t,x,y^t)=h(t)\right\}.
\end{align}
Note that inequality $h(t)\leq g(t)$ holds. If $h(t) = g(t)$ the set of saddle points is given by
\begin{equation}
S(t) := X(t)\times Y(t).
\end{equation}
We state now a simplified version of a result from \cite{MR948649} derived from \cite{CS} which gives realistic conditions that allows to differentiate $g(t)$ at $t=0$.  The main difficulty is to obtain conditions which allow to exchange the derivative with respect to $t$ and the inf-sup in \eqref{s2}.  
\begin{theorem}[Correa and Seeger\cite{DZ}]
\label{CS}
Let $X,Y,G$ and $\varepsilon$ be given as above. Assume that the following conditions hold :
\begin{itemize}
\item[(H1)] $S(t)\neq\emptyset$ for $0\leq t\leq \varepsilon$.
\item[(H2)] The partial derivative $\partial_t G(t,x,y)$ exists for all $(t,x,y)\in [0,\varepsilon]\times X\times Y $.
\item[(H3)] For any sequence $\{t_n\}_{n\in\N}$, with  $t_n\to 0$, there exist  a subsequence $\{t_{n_k}\}_{k\in\N}$ and $x^0\in X(0)$, $x_{n_k}\in X(t_{n_k})$ such that for all $y\in Y(0)$,
\begin{equation*}
\lim_{t\searrow 0, k\to\infty} \partial_t G(t,x_{n_k},y) = \partial_t G(0,x^0,y),
\end{equation*}
\item[(H4)] For any sequence $\{t_n\}_{n\in\N}$, with  $t_n\to 0$, there exist  a subsequence $\{t_{n_k}\}_{k\in\N}$ and $y^0\in Y(0)$, $y_{n_k}\in Y(t_{n_k})$ such that for all $x\in X(0)$,
\begin{equation*}
\lim_{t\searrow 0, k\to\infty} \partial_t G(t,x,y_{n_k}) = \partial_t G(0,x,y^0),
\end{equation*}
\end{itemize}
Then there exists $(x^0,y^0)\in X(0)\times Y(0)$  such that
\begin{align*}
\frac{dg}{dt}(0) &=  \partial_t G(0,x^0,y^0).
\end{align*}
\end{theorem}

\bibliographystyle{siam}
\bibliography{biblio}

\end{document}